\documentclass[a4paper]{article}
 
\usepackage{amsmath}
\usepackage{amsthm}
\usepackage{amsfonts}
\usepackage{graphicx}
\usepackage{subfig}

\providecommand{\R}{\mathbb{R}}
\providecommand{\p}{\mathbb{P}^2}
\providecommand{\face}[1]{\textit{#1}}
\providecommand{\ident}{\leftrightarrow}
\providecommand{\Link}{\textrm{\em Link}}
\providecommand{\simtimes}{\mathbin{%
  \stackrel{\sim}{\smash{\times}\rule{0pt}{0.6ex}}%
}}

\theoremstyle{definition}
\newtheorem{theorem}{Theorem}
\newtheorem{lemma}[theorem]{Lemma}
\newtheorem{corollary}[theorem]{Corollary}
\newtheorem{definition}[theorem]{Definition}
\newtheorem{example}[theorem]{Example}
\newtheorem{notation}[theorem]{Notation}
\newtheorem{improvement}[theorem]{Improvement}
\theoremstyle{construction}
\newtheorem{con}[theorem]{Construction}

\newcounter{theorem-store}
\newcounter{edge-reverse-store}
\newcounter{edge-count-store}

\title{An edge-based framework for enumerating 3-manifold triangulations%
    \footnote{Partially supported by the Australian Research Council
    (projects DP1094516, DP110101104).}}

\author{Benjamin A. Burton\\
School of Mathematics and Physics, The University of Queensland,\\
  Brisbane QLD 4072, Australia\\
  \texttt{bab@maths.uq.edu.au} \and
William Pettersson\\
School of Mathematics and Physics, The University of Queensland,\\
  Brisbane QLD 4072, Australia\\
  \texttt{william@ewpettersson.se}}

\begin{document}

\maketitle

\begin{abstract}
A typical census of 3-manifolds contains all manifolds (under various
constraints) that can be triangulated with at most $n$ tetrahedra. Although
censuses are useful resources for mathematicians, constructing them is
difficult: the best algorithms to date have not gone beyond $n=12$. The
underlying algorithms essentially (i) enumerate all relevant 4-regular
multigraphs on n nodes, and then (ii) for each multigraph $G$ they enumerate
possible 3-manifold triangulations with $G$ as their dual 1-skeleton, of which
there could be exponentially many. In practice, a small number of multigraphs
often dominate the running times of census algorithms: for example, in a
typical census on 10 tetrahedra, almost half of the running time is spent on
just 0.3\% of the graphs. 

Here we present a new algorithm for stage (ii), which is the computational
bottleneck in this process. The key idea is to build triangulations by
recursively constructing neighbourhoods of edges, in contrast to traditional
algorithms which recursively glue together pairs of tetrahedron faces. We
implement this algorithm, and find experimentally that whilst the overall
performance is mixed, the new algorithm runs significantly faster on those
``pathological'' multigraphs for which existing methods are extremely slow. In
this way the old and new algorithms complement one another, and together can
yield significant performance improvements over either method alone.
\end{abstract}

\section{Introduction}
In many fields of mathematics, one can often learn much by studying an exhaustive ``census'' of certain objects,
such as knot dictionaries.
Our focus here is on censuses of closed 3-manifolds---essentially topological spaces that locally look like $\R^3$.
Combinatorially, any closed 3-manifold can be represented by a \emph{triangulation}, formed from tetrahedra with faces identified together in pairs
\cite{Moise1952}.  A typical census of 3-manifolds enumerates all 3-manifold under certain conditions that can be constructed from a fixed
number of tetrahedra.

One of the earliest such results was a census of all cusped hyperbolic 3-manifolds which could be built from at most 5 tetrahedra, by Hildebrand and Weeks \cite{Hildebrand1989};
this was later extended to all such manifolds on at most 9 tetrahedra \cite{Burton2014Cusped,Callahan1999,thistlethwaite10-cusped8}.
For closed orientable 3-manifolds, Matveev gave the first census of closed orientable prime manifolds on up to 6 tetrahedra \cite{Matveev1998};
this has since been extended to 12 tetrahedra \cite{Martelli2001,Matveev2007AlgorithmicTopology}.

Most (if not all) census algorithms in the literature enumerate 3-manifolds on $n$ tetrahedra in two main stages.
The first stage is to generate a list of all 4-regular multigraphs on $n$ nodes. 
The second stage takes each such graph $G$, and sequentially identifies faces of tetrahedra together to form a triangulation with $G$ as its dual 1-skeleton (for a highly tuned implementation of such an algorithm, see \cite{Regina}).

There are $|S_3|=6$ possible maps to use for each such identification of faces.
Thus for each graph $G$, the algorithm searches through an exponential (in the number of tetrahedra) search tree, and each leaf in this tree is a triangulation but is not necessarily a 3-manifold triangulation.
Much research has focused on trimming this search tree down by identifying and pruning subtrees which only contain triangulations which are not
3-manifold triangulations \cite{Burton2007,burton11-genus,martelli02-decomp,Matveev1998}.

In this paper we describe a different approach to generating a census of 3-manifolds.
The first stage remains the same, but in the second stage we build up the neighbourhood of each \emph{edge} in the triangulation recursively,
instead of joining together faces one at a time.
This is, in a sense, a paradigm shift in census enumeration,
and as a result it generates significantly different search trees with very different opportunities for pruning.
By implementing the new algorithm and comparing its performance against existing algorithms,
we find that this new search framework complements existing algorithms very well,
and we predict that a heuristic combination that combines the benefits of this with existing algorithms can significantly speed up census enumeration.

The key idea behind this new search framework is to
extend each possible dual 1-skeleton graph to a ``fattened face pairing graph'', and then to find particular cycle-based decompositions of these new graphs.
We also show how various improvements to typical census algorithms (such as those in \cite{Burton2004}) can be translated into this new setting.

\section{Definitions and notation}\label{sec:notation}

In combinatorial topology versus graph theory, the terms ``edge'' and ``vertex'' have distinct meanings.
Therefore in this paper, the terms {\em edge} and {\em vertex} will be used to
mean an edge or vertex of a tetrahedron, triangulation or manifold; and the terms {\em arc}
and {\em node} will be used to mean an edge or vertex in a graph respectively. 

A 3-manifold is a topological space that locally looks like either 3-dimensional Euclidean space (i.e., $\R^3$) or closed 3-dimensional Euclidean half-space (i.e., $\R^3_{z\geq 0}$).
In this paper we only consider compact and connected 3-manifolds.
When we refer to faces, we are explicitly talking about 2-faces (i.e., facets of a tetrahedron).
We represent 3-manifolds combinatorially as triangulations \cite{Moise1952}: a collection of tetrahedra (3-simplices) with some 2-faces pairwise identified.

\begin{definition}
  A {\em general triangulation} is a collection
  $\Delta_1,\Delta_2,\ldots,\Delta_n$ of $n$ abstract tetrahedra, along with
  some bijections $\pi_1,\pi_2,\ldots,\pi_m$ where each bijection 
  $\pi_i$ is an affine map between two faces of tetrahedra, and each face of
  each tetrahedron is in at most one such bijection.
\end{definition}

We call these affine bijections {\em face identifications} or simply {\em identifications}.
Note that this is more general than a simplicial complex (e.g., we allow an identification between two distinct faces of the same tetrahedron).
If the quotient space of such a triangulation is a 3-manifold, we will say that the triangulation represents said 3-manifold.

\begin{notation}
Given a tetrahedron with vertices $a$, $b$, $c$ and $d$, we will define face $a$
to be the face opposite vertex $a$. That is, face $a$ is the face consisting of
vertices $b$, $c$ and $d$. We will sometimes also refer to this as face \face{bcd}.
We will write $\face{abc} \ident \face{efg}$ to mean that face \face{abc} is identified
with face \face{efg} and that in this identification we have vertex
$a$ identified with vertex $e$, vertex $b$ identified with vertex $f$ and
vertex $c$ identified with vertex $g$.

We will also use the notation $ab$ to denote the edge joining vertices $a$ and
$b$ on some tetrahedron. Note that by this notation, the edge $ab$ on a
tetrahedron with vertices labelled $a$, $b$, $c$ and $d$ will be the
intersection of faces $c$ and $d$.
\end{notation}

As a result of the identification of various faces, some edges or
vertices of various tetrahedra are identified together.
The {\em degree} of an edge of the triangulation, denoted $\deg(e)$, is defined to be the number of edges of tetrahedra which are identified together to form the edge of the triangulation.

We also need to define the {\em link} of a vertex before we can discuss 3-manifold triangulations.
An example of a general triangulation, and a link of one of its vertices, is given in Example \ref{ex:gt} below.

\begin{definition}\label{definition:link}
  Given a vertex $v$ in some triangulation, the link of $v$, denoted $\Link(v)$,
  is the (2-dimensional) frontier of a small regular neighbourhood of $v$.
\end{definition}

\begin{example}\label{ex:gt}
  \begin{figure}[h]
  \centering
  \subfloat[]{\label{fig:gt_a}\includegraphics[scale=0.6]{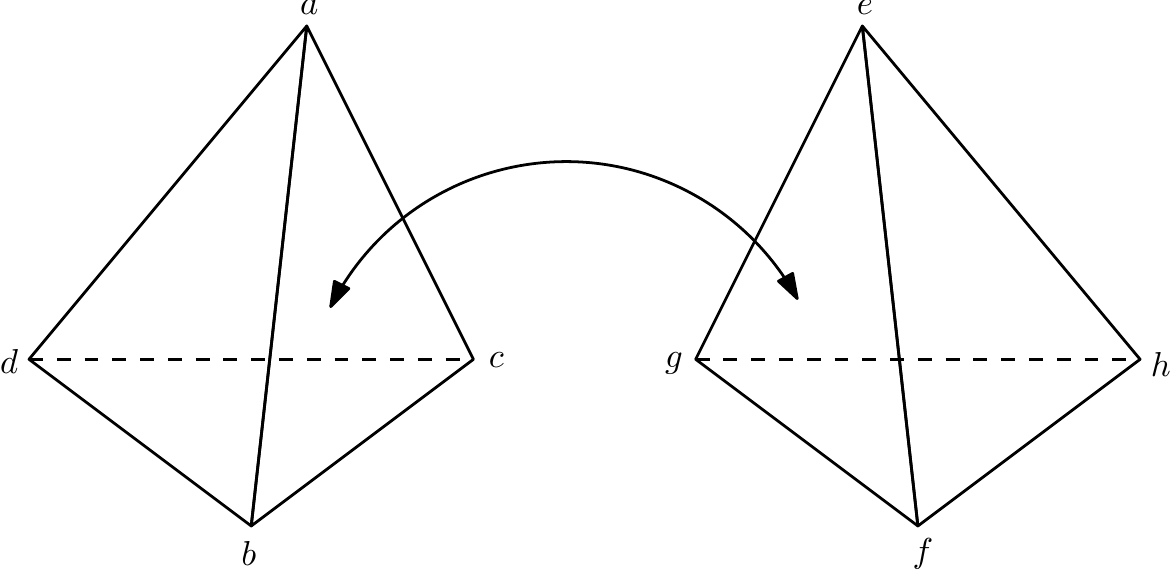}}
  \qquad
  \qquad
  \subfloat[]{\label{fig:gt_b}\includegraphics[scale=0.6]{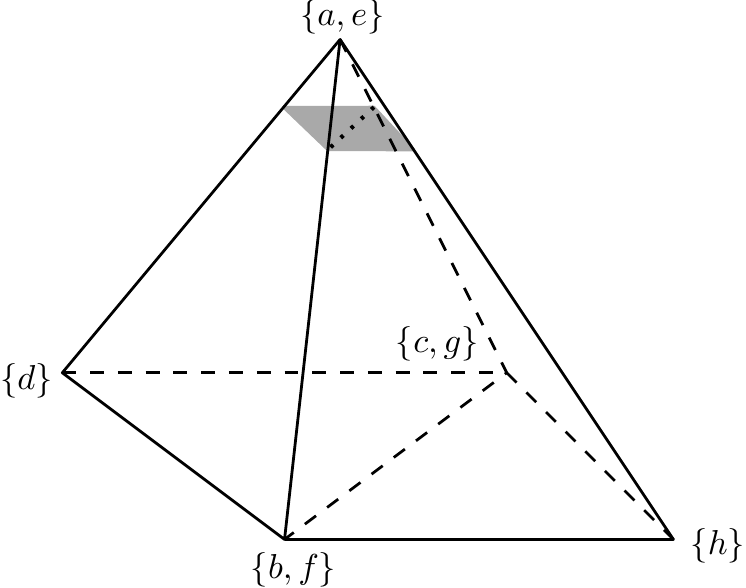}}
  \caption{A general triangulation with one face identification, shown as two
  distinct tetrahedra with an arrow indicating the identification (a); and the
combined triangulation (b). The grey rectangle in (b) indicates the link of the
vertex $\{a,e\}$.}
  \label{fig:gt}
\end{figure}
Consider two tetrahedra; one with
vertices labelled $a$, $b$, $c$, and $d$ and the second with vertices labelled
$e$, $f$, $g$ and $h$; and then apply the face identification $\face{abc} \ident
\face{efg}$.  The resulting triangulation is topologically a 3-ball.
Figure \ref{fig:gt_a} shows this triangulation, with the arrow indicating
two faces being identified. The actual identification involved is not displayed
in the diagram, however.
\end{example}

We now detail the properties a general triangulation must have to represent a 3-manifold. 

\begin{lemma}\label{lemma:3mfld_tri}
A general triangulation is a {\em 3-manifold triangulation}  if
the following additional conditions hold:
\begin{itemize}
  \item the triangulation is connected;
  \item the link of any vertex in the triangulation is homeomorphic to either a
    2-sphere or a disc;
  \item no edge in the triangulation is identified with itself in reverse.
\end{itemize}
\end{lemma}

It is well known that these conditions are both necessary and sufficient for the underlying topological space to be a 3-manifold (possibly with boundary).
However in this paper we only consider with 3-manifolds without boundary. That is, every face of a tetrahedron will be identified with some other face in a 3-manifold triangulation.
Example \ref{ex:3sphere} now gives an example 3-manifold triangulation of a 3-sphere.
\begin{example}\label{ex:3sphere}
\begin{figure}
  \centering
  \includegraphics{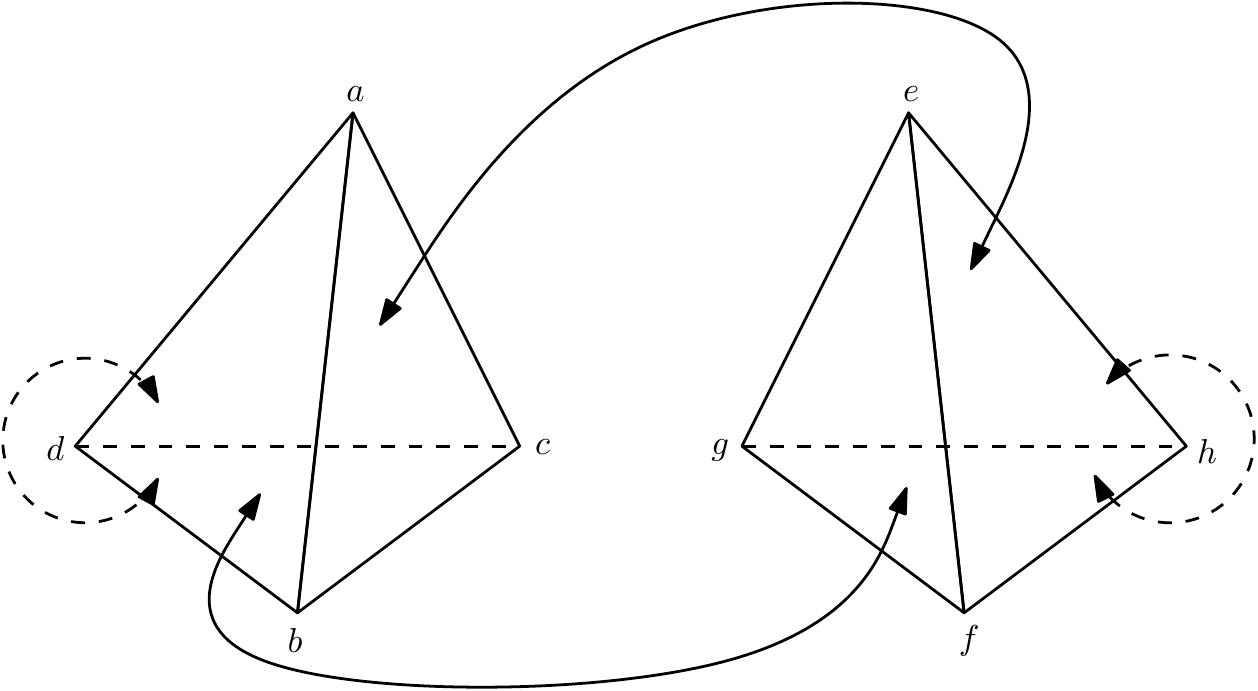}
  \caption{A visual representation of the  3-sphere triangulation described in
    Example \ref{ex:3sphere}. The arrows indicate which faces are identified
  together with dashed arrows referring to the ``back'' faces. Note that some
of the identifications involve rotations or flips which are not shown in the
figure.}
  \label{fig:3sphere}
\end{figure}
Figure \ref{fig:3sphere} shows how two tetrahedra may have faces identified
together to form a triangulation of a 3-sphere\footnote{A 3-sphere is a
  3-manifold that in $\mathbb{R}^4$ can be modelled as $x_1^2 + x_2^2 + x_3^2 +
x_4^2 = 1$. It can be thought of as the 3-dimensional surface of a 4-dimensional ball.}.
Each tetrahedron has two faces identified together, and
another two faces identified with the corresponding faces from the second
tetrahedra. The exact identifications are as follows.
$$
\begin{array}{cc}
  \face{abc} \ident \face{hfe} & \face{abd} \ident \face{gfe} \\
  \face{acd} \ident \face{bcd} & \face{egh} \ident \face{fgh}
\end{array}
$$
\end{example}

We now given some results on the links of vertices in various
triangulations and manifolds. These results are well known, and are given for
completeness. First, however, we need the following definition.

\begin{definition}
  The {\em Euler characteristic} of a triangulation is topological invariant, denoted as $\chi$. For triangulations it can be calculated as $\chi =
  V - E + F - T$ where $V$, $E$, $F$ and $T$ are the number of vertices, edges,
  faces and tetrahedra in the triangulation respectively. For 2-dimensional triangulations, $\chi = V - E + F$.
\end{definition}

For the following proofs, we also briefly need the Euler characteristic of a cell decomposition\footnote{A triangulation is a cell decomposition satisfying some extra properties.}.
We omit the technical details, but $\chi$ can be calculated as $\sum_i (-1)^i k_i$ where $k_i$ is the number of $i$-cells in the decomposition.

It is well known by the classification of 2-manifolds (\cite{Weeks1999ZIP}) that $\chi \leq 2$ for any surface and that $\chi = 2$ if and only if the the surface is a 2-sphere.
Additionally, a closed 3-manifold (that is, a compact 3-manifold with no boundary) has an Euler characteristic of zero (by Poincar\'e duality, see \cite{Hatcher2002Algebraic}).

The following lemmas will help determine the form of links in various
triangulations.

\begin{lemma}\label{lemma:inclusionexclusion}
  Take two triangulations $L$ and $K$ with combinatorially equivalent boundaries, and create a new triangulation $M$ by identifying $L$ and $K$ along their boundaries.
  Then $\chi(M) = \chi(L) + \chi(K) - \chi(\partial L)$.
\end{lemma}
The above follows from a simple counting argument along the shared boundary. We can now prove Lemma \ref{lemma:2sphere-links}.

\begin{lemma}\label{lemma:2sphere-links}
  Given any connected closed triangulation $T$ on $n$ tetrahedra with $k$ vertices where no edge is identified with itself in reverse, the triangulation has $n+k$ edges if and only if the link of each vertex in $T$ is homeomorphic to a 2-sphere.
\end{lemma}
\begin{proof}
Let the triangulation have $e$ edges. As
each face of $T$ is identified with exactly one other face,
and a tetrahedron has 4 faces, we know that $T$ must have $2n$
faces.  Then $\chi(T) = k - e + 2n - n = n+k-e$ which immediately gives one direction of the proof.

Label each vertex in $T$ by $\{v_1,\ldots,v_k\}$.
Now consider the collection of truncated tetrahedra (or 3-cells) $T'$ created from $T$ by truncating $T$ along the link of every vertex $v_i$.
The boundary of $T'$ is therefore the union of the links of the vertices of $T$.
When forming $T'$ from $T$, for each vertex $v_i$ we removed one 0-cell, and then added a cell for every vertex, face and edge of $\Link(v_i)$.
As a result, we get that \[\chi(T') = \chi(T) + \sum_{i=1}^k \left(\chi(Link(v_i)) - 1\right).\]
Furthermore, since $\chi(T) = n+k-e$ we get
\begin{align*}
  \chi(T') = n - e + \sum_{i=1}^k\chi(\Link(v_i)). \tag{$\star$}
\end{align*}

Note that $T'$ represents a compact 3-manifold with boundary as it contains no edge identified with itself in reverse and every vertex on $\partial T'$ has a link homeomorphic to a disc.

Now take a copy of $T'$, call the copy $T''$, and identify the boundary of $T'$ to the boundary of $T''$ to form the triangulation $T^\dagger$. 
Since $T'$ is a compact 3-manifold with boundary, $T^\dagger$ is a closed 3-manifold, giving $\chi(T^\dagger)=0$ (even if the vertex links of $T$ are not spheres).

If we set $L=T'$, $K=T''$ and $M = T^\dagger$ in Lemma
\ref{lemma:inclusionexclusion}, 
then we get the following result.
\[
  0 = \chi(T') + \chi(T') - \sum_{i=1}^k\chi(\Link(v_i))
\]
We then substitute in from $(\star)$ and rearrange to get
\[
  \sum_{i=1}^k \chi(\Link(v_i)) = 2(e-n).
\]
Since $\chi(\Link(v_i)) = 2$ if and only if the link of each vertex is
homeomorphic to a 2-sphere, and is less than $2$ otherwise, we get that
$2k = 2(e-n)$ if and only if the link of each vertex of $T$ is homeomorphic to a
3-sphere.
\end{proof}

We also need to define the {\em face pairing graph} of a triangulation.
The {\em face pairing graph} of a triangulation, also known as the dual
1-skeleton, is a graphical representation of the face identifications of the
triangulation. Each tetrahedron is associated with a node in the face pairing
graph, and one arc joins a pair of tetrahedra for each identification of faces
between the two tetrahedra.  Note that face pairing graph is not necessarily a
simple graph.  Indeed, it will often contain both loops (when there is an
identification of two distinct faces of the same tetrahedron) and parallel arcs (when
there are multiple face identifications between two tetrahedra).

Lastly, we need a few properties of manifolds and triangulations.

\begin{definition}
  A 3-manifold $\mathcal{M}$ is {\em irreducible} if every embedded 2-sphere in $\mathcal{M}$ bounds a 3-ball in $\mathcal{M}$.
\end{definition}

\begin{definition}
  A 3-manifold $\mathcal{M}$ is {\em prime} if it cannot be written as a connected sum of two manifolds where neither is a 3-sphere.
\end{definition}

\begin{definition}
  A 3-manifold is {\em $\mathbb{P}^2$-irreducible} if it is irreducible and also contains no embedded two-sided projective plane.
\end{definition}

Prime manifolds are the most fundamental manifolds to work with.
We note that prime 3-manifolds are either irreducible, or are one of the orientable direct product $S^2 \times S^1$ or the non-orientable twisted product $S^2 \simtimes S^1$.
As these are both well known and have triangulations on two tetrahedra, for any census of minimal triangulations on three or more tetrahedra
we can interchange the conditions ``prime'' and ``irreducible''.
Any non-prime manifold can be constructed from a connected sum of prime manifolds, so enumerating prime manifolds is sufficient for most purposes.
A similar (but more complicated) notion holds for $\mathbb{P}^2$-irreducible manifolds in the non-orientable setting.
As such, minimal prime $\mathbb{P}^2$-irreducible triangulations form the basic building blocks in combinatorial topology.

\begin{definition}
  A 3-manifold triangulation of a manifold $\mathcal{M}$ is {\em minimal} if $\mathcal{M}$ cannot be triangulated with fewer tetrahedra.
\end{definition}

Minimal triangulations are well studied, both for their relevance to computation and for their applications in zero-efficient triangulations \cite{Jaco2003ZeroEfficient}.
Martelli and Petronio \cite{martelli02-decomp} also showed that, with the exceptions $S^3$, $RP^3$ and $L_{3,1}$,
the minimal number of tetrahedra required to triangulate a closed, irreducible and $\mathbb{P}^2$-irreducible 3-manifold $\mathcal{M}$ is equal to the
\emph{Matveev complexity} \cite{Matveev2007AlgorithmicTopology} of $\mathcal{M}$.

\section{Manifold decompositions}\label{sec:decomp}

In this section we define a fattened face pairing graph, and show how we can represent any general triangulation as a specific decomposition of its fattened face pairing graph.
This allows us to enumerate general triangulations by enumerating graph decompositions.
We then demonstrate how to restrict this process to only enumerate 3-manifold triangulations.

A {\em fattened face pairing graph} is an extension of a face pairing graph
$F$ which we use in a dual representation of the corresponding
triangulation. Instead of one node for each tetrahedron, a fattened face
pairing graph contains one node for each face of each tetrahedron. Additionally,
a face identification in the triangulation is represented by \emph{three} arcs in the fattened face
pairing graph; these three arcs loosely correspond to the three pairs of edges which are identified as a consequence of the face identification.
\begin{definition}
Given a face pairing graph $F$, a fattened face pairing graph 
is constructed by first tripling each arc (i.e., for each arc $e$ in $F$, add two more arcs parallel to $e$), and then replacing each node
$\nu$ of $F$ with a copy of $K_4$ such that each node of the $K_4$ is incident with exactly one set of triple arcs that meet $\nu$.

\end{definition}

\begin{example}\label{ex:ffpg}
Figure \ref{fig:ffpg} shows a face pairing graph and the resulting fattened
face pairing graph. The arcs shown in green are what we call
{\em internal} arcs. Each original node has been replaced
with a copy of $K_4$ and in place of each original arc a set of three parallel arcs have
been added.
\begin{figure}[h]
  \centering
  \subfloat[]{\includegraphics[scale=0.7]{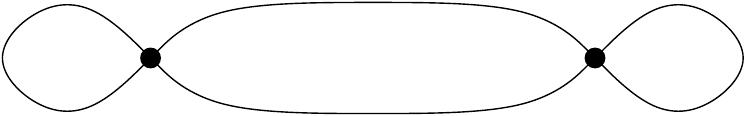}}
  \quad
  \subfloat[]{\includegraphics[scale=0.4]{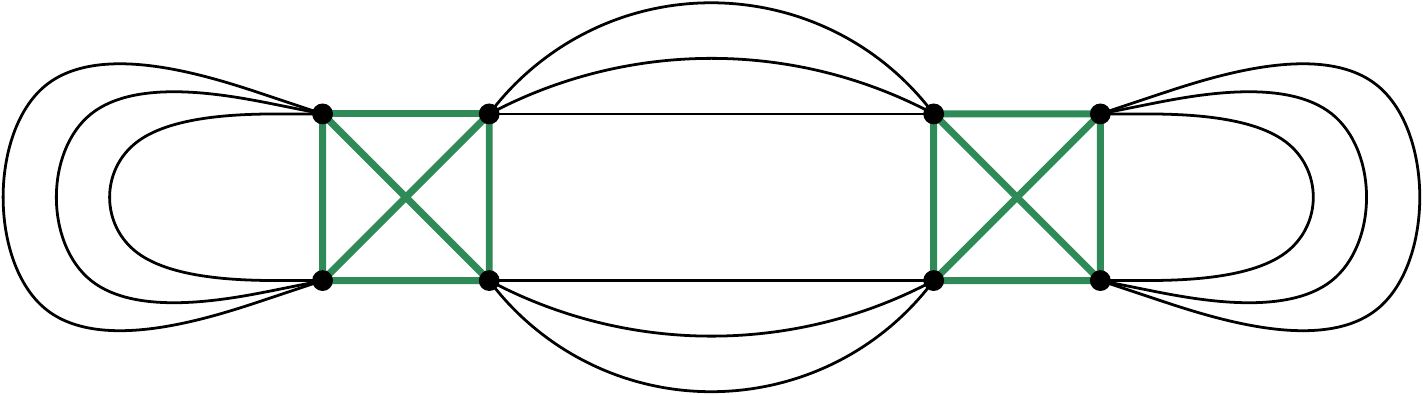}}
  \caption[Fattened face pairing graph]{The face pairing graph (a) and fattened face pairing graph (b) of a
  triangulation. Note that the green arcs are internal arcs, while the
  black arcs are external arcs.}
  \label{fig:ffpg}
\end{figure}
\end{example}

We will refer to the arcs of each $K_4$ as {\em internal arcs}, and the
remaining arcs (coming from the triple edges) as {\em external arcs}.
As a visual aid we will always draw internal arcs in green.
Each such $K_4$
represents a tetrahedron in the associated triangulation, and as such we
will say that a fattened face pairing graph has $n$ tetrahedra if it contains
$4n$ nodes.

Triangulations are often labelled or indexed in some manner.
Given any labelling of the tetrahedra and their vertices, we label the corresponding fattened face pairing graph as follows.
For each tetrahedron $i$ with faces $a$, $b$, $c$ and $d$, we label the nodes of the corresponding $K_4$ in the fattened face pairing graph
$v_{i,a}$, $v_{i,b}$, $v_{i,c}$ and $v_{i,d}$, such that if face $a$ of tetrahedron $i$ is identified with face $b$ of tetrahedron $j$ then there are three parallel external arcs between $v_{i,a}$ and $v_{j,b}$.

In such a labelling, the node $v_{i,a}$ represents face $a$ of tetrahedron $i$.
Each internal arc $\{v_{i,a},v_{i,b}\}$ represents the unique edge common to faces $a$ and $b$ of tetrahedron $i$.
Each external arc $\{v_{i,a},v_{j,b}\}$ represents one of the three pairs of edges of tetrahedra which become identified as a result of identifying face $a$ of tetrahedron $i$ with face $b$ of tetrahedron $j$.
Note that the arc only represents the pair of edges being identified, and does not indicate the orientation of said identification.

We now define {\em ordered decompositions} of fattened face pairing graphs.
Later we show that there
is a natural correspondence between such a decomposition and a general
triangulation,
and we show exactly how the 3-manifold constraints on general triangulations (see Lemma~\ref{lemma:3mfld_tri}) can be translated to constraints on these decompositions.
There is also a natural relationship between such decompositions and \emph{spines} of 3-manifolds, as used by Matveev and others
\cite{Matveev2007AlgorithmicTopology}; we touch on this relationship again later in this section.

\begin{definition}\label{definition:ordered-decomp}
An {\em ordered decomposition} of a fattened face pairing graph $F=(E,V)$ is a
set of closed walks $\{P_1,P_2,\ldots,P_n\}$
such that:
\begin{itemize}
  \item $\{P_1,P_2,\ldots,P_n\}$ partition the arc set $E$;
  \item $P_i$ is a closed walk of even length for each $i$; and
  \item if arc $e_{j+1}$ immediately follows arc $e_j$ in one of the walks then
  exactly one of $e_j$ or $e_{j+1}$ is an internal arc.
\end{itemize}
\end{definition}

An ordered decomposition of a fattened face pairing graph exactly describes a
general triangulation. We first outline this idea here by showing how three parallel external arcs can represent an identification of faces. 
Complete technical details follow later.

Since the ordered decomposition consists of closed walks of alternating internal and external arcs, the decomposition
pairs up the six arcs exiting each nodes so that
each external arc is paired with exactly one internal arc. To
help visualise this, we can draw such nodes
as larger ellipses, with three external arcs and three internal arcs entering the
ellipse, as in Figure \ref{fig:nodezoomed}. Each external arc meets exactly
one internal arc inside this ellipse. This only represents how such arcs are
paired up in a given decomposition---the node is still incident with all six
arcs. We also see in Figure \ref{fig:nodezoomed} that the fattened face
pairing graph can always be drawn such that any ``crossings'' of
arcs only occur between external arcs. Such crossings are simply artefacts of how the fattened face pairing graph is drawn in the plane, and in no way represent any sort of underlying topological twist.

\begin{figure}
  \centering
  \subfloat[]{\includegraphics[scale=0.5]{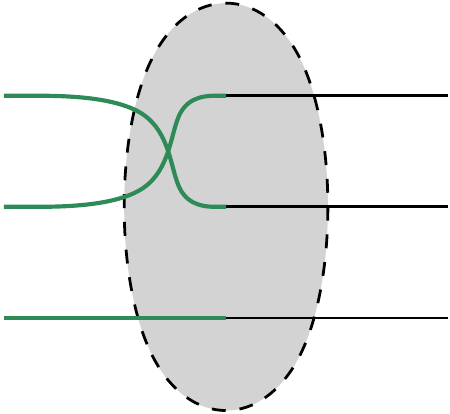}}
  \qquad
  \qquad
  \subfloat[]{\includegraphics[scale=0.5]{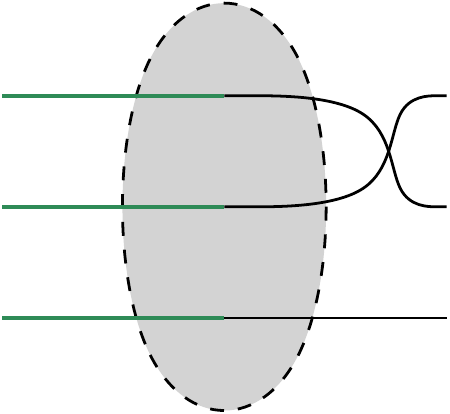}}
  \caption[Close up view of a node in a fattened face pairing graph]{Two close up views of a node of a fattened face pairing graph with
    the same pairing of arcs. The node itself is represented by the grey
    ellipse, and all six arcs are incident upon this node. Note how both figures
    show the same pairing of edges, the only difference is where the ``crossing''
  occurs.}
  \label{fig:nodezoomed}
\end{figure}

Figure \ref{fig:TwoNodes} shows a partial drawing of an ordered decomposition
of a fattened face pairing graph.
In this, we see a set of three parallel external arcs between nodes $v_{1,d}$
and $v_{2,h}$. This tells us that face $d$
of tetrahedron 1 is identified with face $h$ of
tetrahedron $2$. Additionally, we see that one of the external arcs connects
internal arc $\{v_{1,c},v_{1,d}\}$ with internal arc $\{v_{2,g},v_{2,h}\}$. This tells us that edge
$\face{ab}$ of tetrahedron $1$ (represented by $\{v_{1,c},v_{1,d}\}$) is identified
with
edge $\face{ef}$ of tetrahedron $2$ (represented by $\{v_{2,g},v_{2,h}\}$). Since we
know that face $\face{abc}$ is identified with face $\face{efg}$ modulo a possible reflection and/or rotation, this
tells us that
vertex $c$ is identified with vertex $g$ in this face identification. We can
repeat this process for the other paired arcs to see that vertex $a$ is
identified with vertex $e$ and vertex $b$ is identified with vertex $f$. The
resulting identification is therefore $\face{abc} \ident \face{efg}$.

\begin{figure}
  \centering
  \includegraphics[scale=0.6]{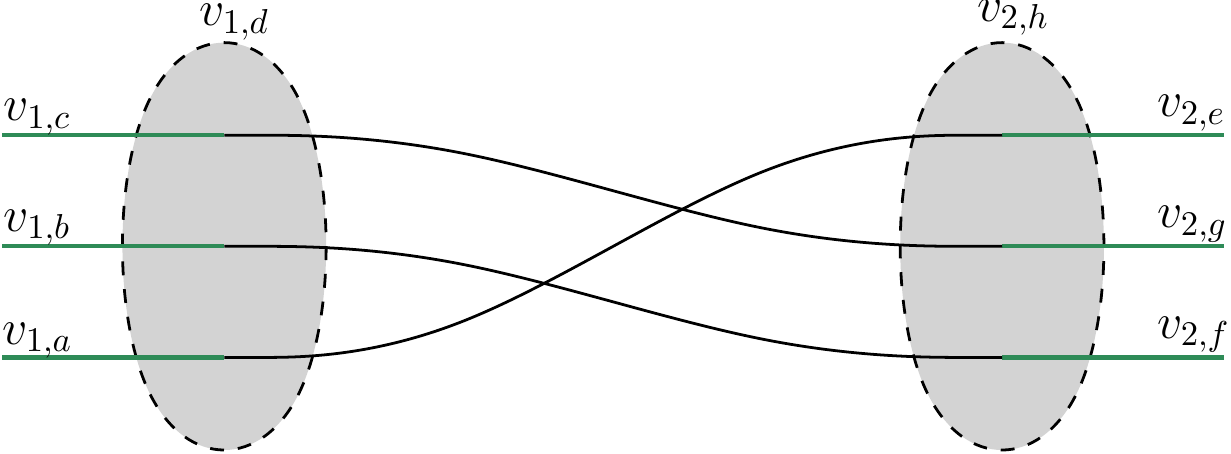}
  \caption{A partial drawing of a fattened face pairing graph.}
  \label{fig:TwoNodes}
\end{figure}

Repeating this for each set of three parallel external arcs gives the required triangulation.
The process is easily reversed to obtain an ordered decomposition from a general triangulation.

We now give the technical details for the construction of an ordered decompositions from a general triangulations, and vice-versa.

\begin{con}[Constructing a general triangulation from an ordered decomposition of a fattened face pairing graph.]\label{lemma:ffpg_to_gtri}

It is straight forward to see that we can simplify a ordered decomposition of
a fattened face pairing graph into a regular face pairing graph, and this gives
a collection of tetrahedra and shows which faces are identified. What remains
is to determine the exact identification between each pair of faces.

First we label the nodes of the fattened face pairing graph such that each $K_4$ in the fattened face pairing graph has nodes labelled $v_{i,a}, v_{i,b}, v_{i,c}, v_{i,d}$.
The choice of $i$ here assigns the label $i$ to the corresponding tetrahedron in the triangulation.
Similarly, the assignment of $v_{i,a}$ to a node labels a face of the corresponding tetrahedron.
Different labellings of nodes will therefore result in a triangulation with different labels on tetrahedra and vertices.
However, up to isomorphism the actual triangulation is not changed.

For each identification of two tetrahedron faces, we have three corresponding external arcs in
the fattened face pairing graph. Each arc $e$ out of these three belongs to one walk in the ordered decomposition, and in said walk $e$ has 
exactly one arc $e_1$ preceding it and one
arc $e_2$ succeeding it such that the sequence of arcs $(e_1,e,e_2)$ occurs the walk.

Since $e$ is an external arc, $e_1$ and $e_2$ must
be internal arcs and therefore of the form $\{v_{i,a},v_{i,b}\}$ where $a\neq b$.
Let $e=\{v_{i,b},v_{j,c}\}$, $e_1 = \{v_{i,a},v_{i,b}\}$ and $e_2
=\{v_{j,c},v_{j,d}\}$.  This tells us that this identification is between face
$a$ of tetrahedron $i$ and face $d$ of tetrahedron $j$, and in
this identification the edge
common to faces $a$ and $b$ on tetrahedra $t_i$ is identified with the edge
common to faces $c$ and $d$ on tetrahedra $t_j$.
The orientation of this edge identification is not given, however it is not needed..
Each of faces $a$ and $d$ have three vertices, and this identification of edges also identifies two vertices from face $a$ with two vertices from face $d$.
This leaves one vertex from each face, which must be identified together.
By repeating this process for the two external arcs parallel to $e$ we can therefore determine the actual face identification between face $a$ and face $d$.

\end{con}

\begin{example}
\begin{figure}
  \centering
  \includegraphics[scale=0.8]{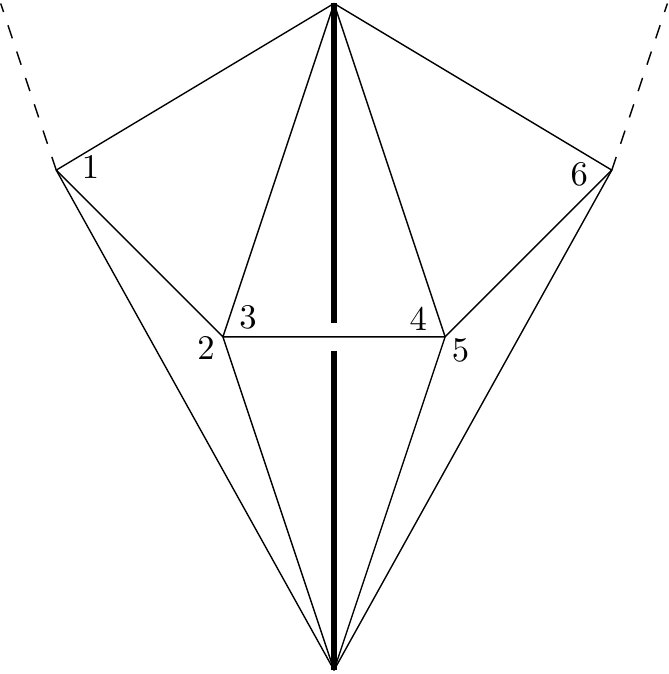}
  \caption{Three tetrahedra about a central edge. Note that only vertices of tetrahedra are labelled in this diagram (i.e.,~vertices 2 and 3 are both of tetrahedra, but in the triangulation they are identified together), and recall that vertex $1$ is opposite face $1$.}
  \label{fig:FaceIdentToDecompExample}
\end{figure}
  First we give an example of how to partially build an ordered decomposition.
  For this example we have a triangulation edge of degree $\geq 3$, depicted in Figure \ref{fig:FaceIdentToDecompExample} as the thicker central edge.
  Recall that face $x$ of a tetrahedron is the face opposite vertex $x$.
  We see that in the leftmost tetrahedron, the thickened edge is opposite vertices $1$ and $2$, and that
  face $1$ is identified with face $4$.
  We therefore have the sequence $(\{v_{1,1},v_{1,2}\},\{v_{1,1},v_{2,4}\},\ldots)$ occurring in one of the walks of the ordered decomposition.
  
  Continuing this process shows that the sequence 
  \[(\{v_{1,1},v_{1,2}\},\{v_{1,1},v_{2,4}\},\{v_{2,4},v_{2,3}\},\{v_{2,3},v_{3,6}\},\{v_{3,6},v_{3,5}\},\ldots)\]
  occurs in one of the walks of the ordered decomposition.
\end{example}

\begin{con}[Constructing an ordered decomposition from a general triangulation.]\label{lemma:gtri_to_ffpg}

First construct the fattened face pairing graph from the face pairing graph of the triangulation.
We now label the fattened face pairing graph. Begin by labelling the tetrahedra in the triangulation, and their vertices.
Label the individual nodes of the fattened face pairing graph such that if face $a$ of tetrahedron $i$ is identified with face $b$ of tetrahedron $j$ then the corresponding three parallel arcs are between node $v_{i,a}$ and node $v_{j,b}$ in the fattened face pairing graph.

Recall that an edge $ab$ is the edge between vertices $a$ and $b$. Given a tetrahedron with vertices labelled $a$, $b$, $c$ and $d$, the edge $ab$ has as endpoints the two vertices $a$ and $b$ and thus is the intersection of face $c$ and face $d$, so the edge $ab$ in the triangulation is represented by the arc $\{v_{i,c},v_{i,d}\}$ in the fattened face pairing graph.

Start with an edge $ab$ on tetrahedron $i$ in the triangulation, and
add $\{v_{i,c},v_{i,d}\}$ to the start of what will become a walk in the ordered
decomposition.

\begin{figure}
  \centering
  \includegraphics[scale=0.8]{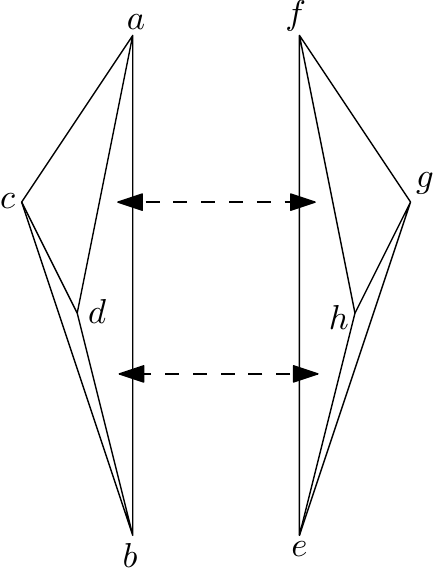}
  \caption{Face $c$ of tetrahedron $i$ is identified with face $g$ of tetrahedron $j$. As a result, one of the walks of the ordered decomposition contains the three arcs $\{\{v_{i,d},v_{i,c}\},\{v_{i,c},v_{j,g}\},\{v_{j,g},v_{j,h}\}$ in order.}
  \label{fig:FaceIdentToDecomp}
\end{figure}

Face $c$ on this tetrahedron must be identified with some face $g$ on tetrahedron
$j$. For a diagram, see Figure \ref{fig:FaceIdentToDecomp}. Through this
identification, the edge $ab$ must be identified with some edge on face
$g$.  Call this edge $ef$. Add one of the three arcs $\{v_{i,c},v_{j,g}\}$ to
the current walk. Since a face contains three edges, by construction
we can always find such an arc which is not already in one of the walks of the ordered decomposition.
If $\{v_{j,g},v_{j,h}\}$ is already in this walk then we are finished with the walk.
Otherwise, add the arc $\{v_{j,g},v_{j,h}\}$ into the walk.
The process then continues with the edge $ef$.  Since each tetrahedron edge is the
intersection of two faces of a tetrahedron, it is clear that this process will
continue until the initial edge $ab$ is
reached and the current walk is complete.

The above procedure is then repeated until all arcs have been added to a
walk. By construction, we have created an ordered decomposition with
the required properties.
\end{con}

  Recall that $\deg (e)$ is the number of edges of tetrahedra identified
  together to form edge $e$ in the triangulation.  The following corollary follows immediately from the constructions.

\begin{corollary}\label{cor:edge_link}
  Given an ordered decomposition $\{P_1,\ldots,P_t\}$, each
  walk $P_i$ corresponds to exactly one edge $e$ in the
  corresponding general triangulation. In addition, $|P_i| = 2 \deg(e)$.
\end{corollary}

Recall that in a 3-manifold triangulation, no edge may be identified with itself in reverse.
In the
triangulation one may consider the ring of tetrahedra
$\Delta_1,\ldots,\Delta_k$ (which need not be distinct) around an edge
$e=ab$, as in Figure \ref{fig:edge_marking_tetrahedra}. Start on $\Delta_1$,
and mark one edge incident to $e$ (say
$bc$) as being ``above'' $e$. Since $bc$ is ``above'' $e$, the face
$bcd$ must be the ``top'' face of $\Delta_1$, and thus the edge $bd$
must also be ``above'' $e$ and is marked. We can then track the edge $bd$
through a face identification, and across the top of the next tetrahedron. At
some point, we must reach $\Delta_1$ again. If $\Delta_1$ is reached via one of
the edges $ac$ or $ad$, then $e$ is identified with itself in
reverse. However, if $\Delta_1$ is reached via the edge $bc$ again, then
we know that the edge $ab$ is not been identified with itself in reverse.

\begin{figure}
  \centering
  \subfloat[Marking process]{\includegraphics[scale=0.8]{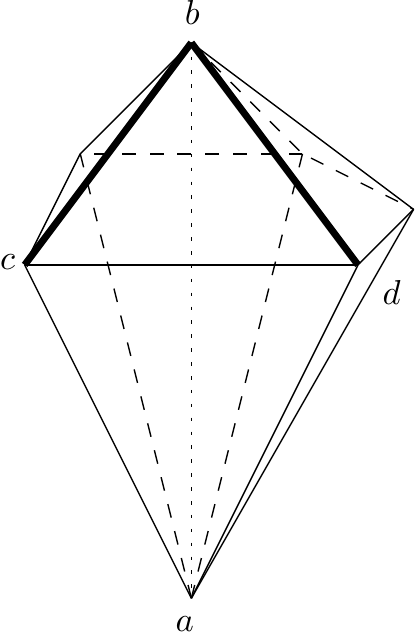}}
  \qquad
  \subfloat[A ``good'' marking]{\includegraphics[scale=0.8]{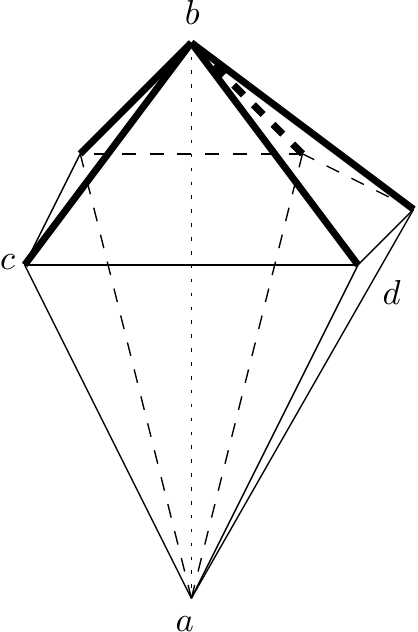}}
  \qquad
  \subfloat[A ``bad'' marking]{\includegraphics[scale=0.8]{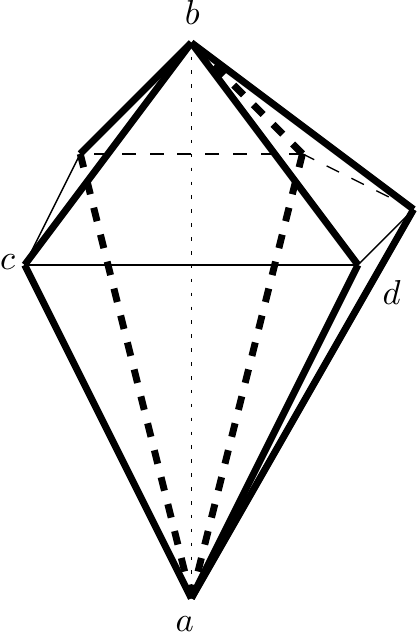}}
  \caption{In these figures, the thicker edges are marked. In (a), edge $bc$
    was arbitrarily marked as being above edge $ab$. Edge $bd$ is marked
    because $bc$ and $bd$ share a common face which does not contain $ab$. If
    the marking in (b) is reached, then the edge $ab$ is not identified with
    itself in reverse. If, however, (c) is reached, then $ab$ is identified with
    itself in reverse.}
  \label{fig:edge_marking_tetrahedra}
\end{figure}

Loosely speaking, in the decomposition setting, we look at one walk $P_x$
of our ordered decomposition and mark arcs in the decomposition as being
``above'' arcs in the walk $P_x$. If we again consider the edge $bc$ as
``above'' $ab$, we mark the arc\footnote{Recall that
  the arc $\{v_{i,a},v_{i,d}\}$ denotes the edge common to face $a$ and face
  $d$. Since face $a$ contains vertices $b$,$c$ and $d$ and face $d$ contains
vertices $a$, $b$ and $c$ this edge must be $bc$.} $\{v_{i,a},v_{i,d}\}$. Since
the ordered decomposition corresponds to exactly one triangulation, we can use
the ordered decomposition to determine which edge is identified with
$\{v_{i,a},v_{i,d}\}$.  We then mark the next edge, and proceed as in the
previous paragraph.
The following definition combined with Lemma \ref{lemma:no_single_edge_reverse} achieves the same result in our new framework.

\begin{definition}\label{definition:marking}
Given an ordered decomposition $\mathcal{P} = \{P_1,P_2,\ldots,P_t\}$, we can
\emph{mark} a walk $P_x$ as follows.
\begin{figure}
  \centering
  \includegraphics[scale=1.0]{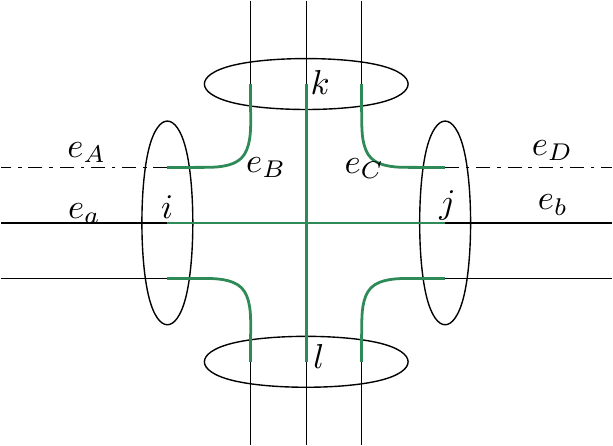}
  \caption{The process used to mark edges as per Definition
    \ref{definition:marking}. The dot-dashed arcs are the ones marked as ``above''. Recall that
  the ellipses are whole nodes, the insides of which denote how internal and
external arcs are paired up in the decomposition.}
  \label{fig:markingedges}
\end{figure}

Pick an external arc $e_s$ from $P_x$. Arbitrarily pick an external arc $e_S$
parallel to $e_s$, and mark $e_S$ as being ``above'' $e_s$.
Then let $e_a = e_s$ and $e_A = e_S$ and continue as follows
(see Figure \ref{fig:markingedges} for a diagram of the construction):

\begin{itemize}
  \item Let $e_b$ be the next external arc in $P_x$ after $e_a$. 
  \item The internal arc preceding $e_b$ joins two nodes. Call these nodes $i$
    and $j$, such that $e_b$ is incident on $j$.
  \item Some external arc $e_A$ incident on $i$ must be marked as ``above'' $e_a$. Find the
    closed walk which $e_A$ belongs to. In this closed walk there must exist
    some internal arc which either
    immediately precedes or follows $e_A$ through node $i$. Call this internal
    arc $e_B$.
    Note that the walk
    containing these two arcs need not be, and often is not, $P_x$. Arc $e_B$ must be incident to
    $i$, and some other node which we shall call $k$.
  \item Find the internal arc $e_C$ between nodes $k$ and $j$, and find the walk
    $P_y$ that it belongs to. In this walk, one of the arcs parallel to $e_b$ must
    either immediately precede or follow $e_C$ and be incident upon node $j$. Call this arc
    $e_D$.
  \item If $e_b = e_s$, and $e_D$ is already marked as being above $e_b$, we
  terminate the marking process.
  \item Otherwise, mark the arc $e_D$ as being above $e_b$ and repeat the above steps, now using $e_b$ in place of $e_a$,
    and using $e_D$ in place of $e_A$.
\end{itemize}
\end{definition}

Note that this process of marking specifically marks one arc as being ``above'' another.
It does not mark arcs as being ``above'' in general.

To visualise this definition in terms of the decomposition, see Figure
\ref{fig:markingedges}. The arcs $e_a$ and $e_b$ are part of a closed walk, and
we are marking the edges ``above'' this walk. Arc $e_A$ was arbitrarily
chosen. Arc $e_B$ follows $e_A$, and then we find $e_C$ as the arc sharing one
node with $e_B$ and one with $e_b$. From $e_C$ we can find and mark $e_D$.

We will show how to achieve a similar result in our new framework.
First we give a overview of the idea, with technical details to follow.
In brief, the walks containing $e_A$ and $e_D$ represent edges of tetrahedra in the triangulation
that share triangles with the common edge represented by $P_x$,
and which both sit ``above'' this common edge (assuming some up/down orientation).
Both $e_B$ and
$e_C$ are internal arcs of the same tetrahedron and share a common node $k$,
so we know that both these internal arcs represent edges of the same tetrahedron
which share a common face $k$.
The external arcs $e_A$ and $e_D$
represents identifications of $e_B$ and $e_C$ respectively
with edges of (typically different) adjacent tetrahedra.

\begin{lemma}\label{lemma:no_single_edge_reverse}
Take an ordered decomposition containing a walk $P_x$ with arcs marked
according to Definition \ref{definition:marking}, and consider the corresponding triangulation.
Then the edge of the triangulation represented by $P_x$ is identified to itself in reverse
if and only if there exists some external arc $e$ in $P_x$ that has two distinct external arcs both marked as ``above'' $e$.
\end{lemma}
\begin{proof}
Part of an ordered decomposition is shown in Figure \ref{fig:BigDecomp}, and we use the notation as shown there.
The part shown represents a single face identification between two (not necessarily distinct) tetrahedra.
The markings on the tetrahedra denote exactly what each labelled arc in the fattened face pairing graph represents.
As such, we say that an internal arc of the ordered decomposition ``is'' also an edge of a tetrahedron.
For example, $e_C$ is an internal arc, so it represents the edge of the tetrahedron yet we say that $e_C$ is the edge on the tetrahedron.
The external arcs all represent edges in face identifications, and are drawn with dashed lines.

\begin{figure}
  \centering
  \includegraphics[scale=0.8]{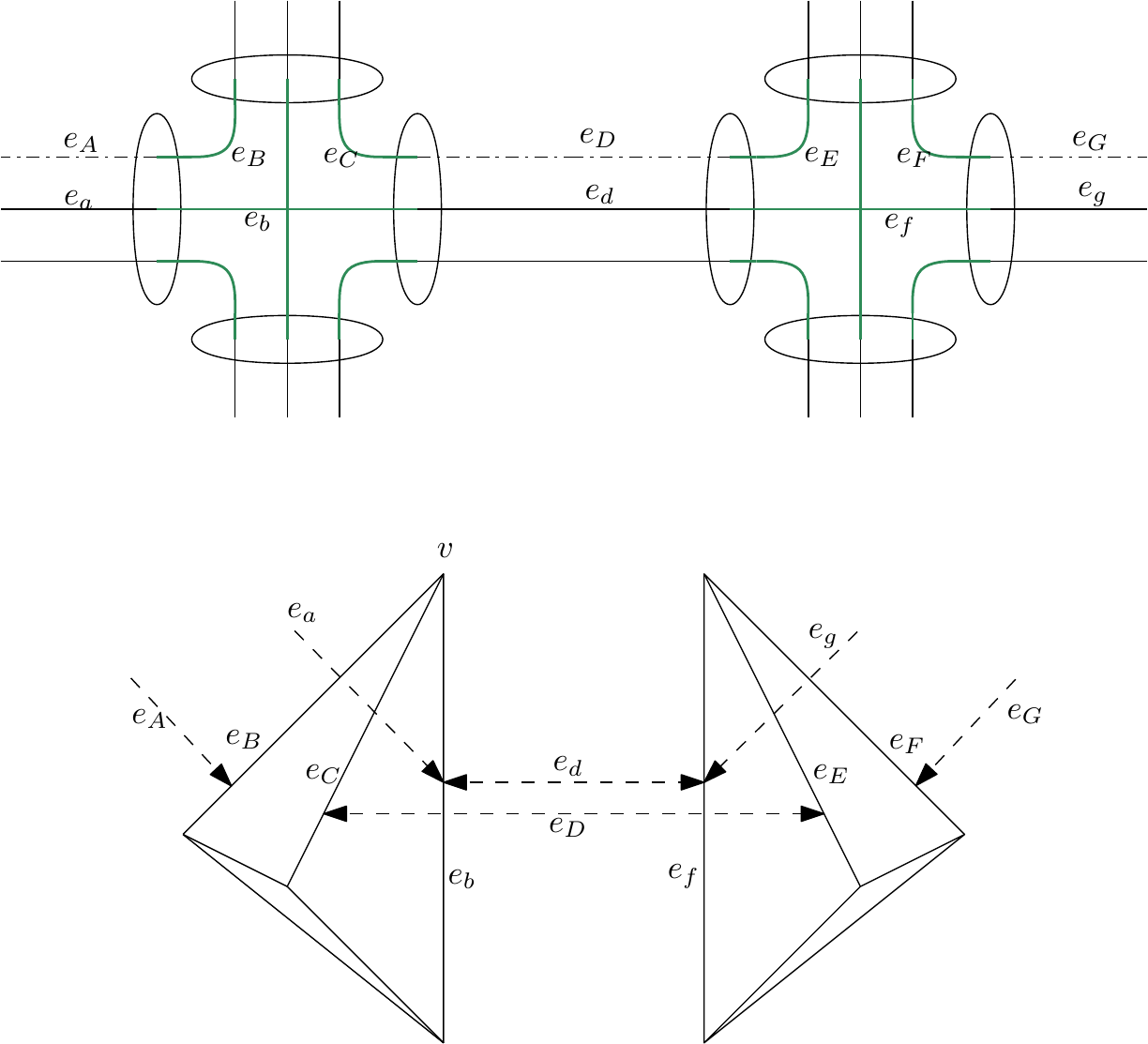}
  \caption{Part of an ordered decomposition, and associated tetrahedra. Identifications of edges are shown with dashed arrows.}
  \label{fig:BigDecomp}
\end{figure}

We prove the result by applying an orientation onto each of the edges of tetrahedra contained in the edge of the triangulation represented by $P_x$.
Consider first the arc $e_a$, which represents one edge identification in some face identification.
The arc $e_A$ (one of the two arcs parallel to $e_a$) is marked as being ``above'' $e_a$.
This is equivalent to assigning an orientation onto each of the pair of edges represented by $e_a$.
Since $e_b$ is one of these, we now have an orientation on the edge $e_b$.
We want to fix an orientation onto the edge $e_f$ such that the orientations of $e_b$ and $e_f$ agree after the identification of faces.
Since $e_B$ immediately follows $e_A$ (or vice-versa) and $e_C$ immediately follows $e_D$ (or vice-versa, again) in the ordered decomposition, edges $e_B$ and $e_C$ meet in a common tetrahedron vertex, call this vertex $v$.
We also see that the edge $e_b$ meets $v$. 
Since the edge $e_b$ is identified with the edge $e_f$ (via the edge identification represented by $e_d$), and the edge $e_C$ is identified with the edge $e_E$ (via the edge identification represented by $e_D$), $v$ must be identified to the vertex common to edges $e_f$ and $e_E$.

The orientation of the edge represented by $e_b$ has been used to orient the edge represented by $e_f$ such that the two orientations agree after the face identification.
Repeating this process for all arcs in $P_x$ in turn orients all the edges of tetrahedra that are contained in the edge of the triangulation.

If every external arc $e$ in $P_x$ has exactly one external arc marked as ``above'' $e$, then we have exactly one orientation for each edge of a tetrahedron.
That is, the edge of the triangulation corresponding to $P_x$ cannot be identified with itself in reverse.

If some external arc $e$ in $P_x$ has two distinct external arcs marked as ``above'' $e$, then every external arc must have two such other arcs marked (as the marking process can only terminate when it reaches $e_s$ in Definition \ref{definition:marking}).
This must mean that we have assigned two distinct orientations to each tetrahedron edge in the triangulation edge corresponding to $P_x$ and therefore this triangulation edge is identified with itself in reverse.
\end{proof}

If a walk $P_x$ in an ordered decomposition satisfies Lemma
\ref{lemma:no_single_edge_reverse}, we say that this walk is 
{\em non-reversing}.

\begin{definition}\label{definition:manifold_decomp}
A {\em manifold decomposition} is an ordered decomposition of a fattened face
pairing graph satisfying all of the following conditions.
\begin{itemize}
  \item The ordered decomposition contains $n+1$ closed walks.
  \item The fattened face pairing graph contains $4n$ nodes.
  \item Each walk is non-reversing.
  \item The associated manifold triangulation contains exactly $1$ vertex.
\end{itemize}
\end{definition}

\begin{theorem}\label{thm:md_equiv_3tri}
  Up to relabelling,
  there is a one-to-one correspondence between manifold decompositions of
  connected fattened face pairing graphs and 1-vertex 3-manifold triangulations.
\end{theorem}
\begin{proof}
Constructions \ref{lemma:ffpg_to_gtri} and \ref{lemma:gtri_to_ffpg} 
give the
correspondence between general triangulations and ordered decompositions. All that
remains is to show that the extra properties of a manifold decomposition force
the corresponding triangulation to be a 3-manifold triangulation.
Since the decomposition contains $n+1$ walks, Corollary
\ref{cor:edge_link} tells us the triangulation has $n+1$ edges. Additionally,
each tetrahedron corresponds to four nodes in the fattened face pairing graph, so the
triangulation has $n$ tetrahedra and thus by
Lemma \ref{lemma:2sphere-links} we see that the link of each vertex is
homeomorphic to a 2-sphere. Each walk is non-reversing so Lemma
\ref{lemma:no_single_edge_reverse} says that no edge in the corresponding
triangulation is identified with itself in reverse, and we have the required
result.
\end{proof}

We now define the notation used to express specific ordered decompositions.
The notation is defined such that it can also be interpreted as a {\em spine code} (as used by Matveev's Manifold Recognizer \cite{MatveevRecog}), and that the spine generated from such a spine code is a dual representation of the same combinatorial object represented by the manifold decomposition.
For more detail on spine codes, see \cite{Matveev2007AlgorithmicTopology}.

\begin{notation}\label{not:md}
  Take an ordered decomposition of a fattened face pairing graph with $4n$ nodes,
  and label each set of three parallel external arcs with a distinct value taken from the set $\{1,\ldots,2n\}$
  (so two external arcs receive the same label if and only if they are part of
  the same triple of parallel arcs).
  Assign an arbitrary orientation to each set of three parallel external arcs.
  For each walk in the ordered decomposition:
  \begin{enumerate}
    \item Create an empty ordered list.
    \item Follow the external arcs in the walk.
      \begin{enumerate}
        \item If an external arc is traversed in a direction consistent with its orientation, add $+i$ to the end of the corresponding ordered list.
        \item If instead the arc in the walk is traversed in the reverse direction, add $-i$ to the end of the list.
        \item Continue until the first external arc in the walk is reached.
      \end{enumerate}
  \end{enumerate}
\end{notation}

See Example \ref{ex:3sphere-ffpg} for an example of the use of this notation.
Note that this notation only records the external arcs, and does not record any internal arcs in walks.

We can also reconstruct the face pairing graph (and therefore the fattened face
pairing graph) from this notation (in particular, we can reconstruct the
internal arcs).
The method essentially uses the fact that
each external arc represents some identification of two faces (and three
parallel external arcs will represent the same identification of two faces),
and so we can use the
orientation of each arc to distinguish between the two faces in each
identification and thereby build up the face pairing graph.

\begin{example}
  \label{ex:3sphere-ffpg}

The following set of walks (remember, we omit internal arcs and instead
prescribe orientations on external arcs) describes a manifold decomposition of
a 3-sphere.

\[ T = \{ (1), (1,2,4,-2,3,-4,-3,-1,3,-2), (4)\} \]

Figure \ref{fig:3sphere-decomp} shows this manifold decomposition of a
3-sphere. Given the appropriate vertex labellings, this represents the same
triangulation as that given in Example \ref{ex:3sphere}.

Each integer in $T$ represents an identification of faces, and we can also track each face in an identification individually using the sign of said integer.
For example, $-3$ is before $-1$ in the second walk, so we can say
that the ``second'' face in identification $1$ belongs to the same
tetrahedron as the ``first'' face in identification $3$. Each integer (or
its negative) appears exactly three times in an ordered decomposition, so we
can determine exactly which faces belong to the same tetrahedron. For example,
both faces involved in identification $1$ belong to the same tetrahedron as the
``first'' face in identification $2$ and the ``first'' face in identification $3$.

\begin{figure}
  \centering
  \includegraphics{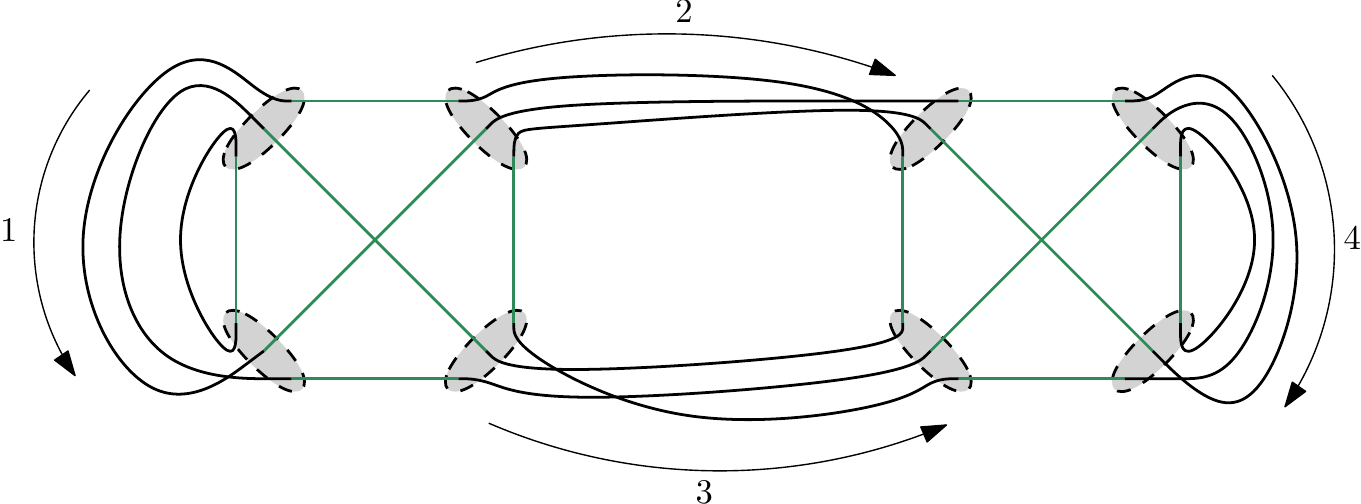}
  \caption{An edge-link decomposition of a fattened face pairing graph that
  represents a 3-sphere. Recall that each grey ellipse is actually a node in
  the fattened face pairing graph.}
  \label{fig:3sphere-decomp}
\end{figure}
\end{example}

An implementation note:
it is trivial, given a fattened face pairing graph and a
``partial'' ordered decomposition in which all the internal arcs are missing,
to reconstruct the complete ordered decomposition.
For the theoretical discussions in this paper we work with the full ordered
decompositions, but in the implementation we only store the sequential list of external arcs as in Notation \ref{not:md}.

\section{Algorithm and improvements} \label{sec:existing-results}
In this section we give various improvements that may be used when enumerating
manifold decompositions (i.e., 3-manifold triangulations).
These are based on known theoretical results in 3-manifold topology, combined with suitable data structures.

Many existing algorithms in the literature
\cite{Regina,Matveev2007AlgorithmicTopology} build triangulations by identifying
faces pairwise (or taking combinatorially equivalent steps, such as annotating
edges of special spines \cite{Matveev2007AlgorithmicTopology}).
The algorithm we give here essentially constructs the neighbourhood of each
\emph{edge} of the triangulation one at a time.
Therefore the search tree traversed by our new algorithm is significantly different than that traversed by other algorithms.
This is highlighted experimentally by the results given in Section \ref{sec:results}.

\subsection{Algorithm}\label{sec:algo}
The basis of our implementation is a simple backtracking approach to enumerate manifold decompositions.
First each set of three parallel arcs is given an arbitrary orientation.
Then each set of three parallel arcs is given a distinct value from the set
${1,\ldots,n}$ (so two external arcs receive the same label if and only if they
are part of the same triple of parallel arcs).
The algorithm then finds the arc $e$ with the lowest such value that is not already in some walk.
A new walk is started with $e$ being used in the ``forwards'' direction.
As seen in Figure \ref{fig:algorithm}, $e$ is incident to some node $n_1$ at its ``head''. There are three choices for the next arc in the walk, corresponding to the three internal arcs incident to $n_1$. The algorithm will try each of these three, one at a time, as long as said internal arc has not already been used in a walk.
Assume internal arc $i$ is used, and that $i$ is the arc between nodes $n_1$ and $n_2$. Then the next step the algorithm takes is to check whether this can be the last arc in the walk. That is, was the first initial arc also incident upon $n_2$. If so, finish this walk and attempt to create the next walk (or if all arcs have been used, determine if the current set of walks comprises an ordered decomposition).
After attempted to complete the walk (and regardless of whether the walk could be completed or not), one of the three parallel arcs incident to $n_2$, say $e_2$, is used next in the walk. Since these are parallel, there is no need to differentiate between them. Additionally, since the internal arc $i$ was not used in a walk, we know that at least one of these three arcs has not been used in a walk.
The algorithm then adds $e_2$ to the current walk and continues the process until all possible ordered decompositions have been found.

\begin{figure}
  \centering
  \includegraphics{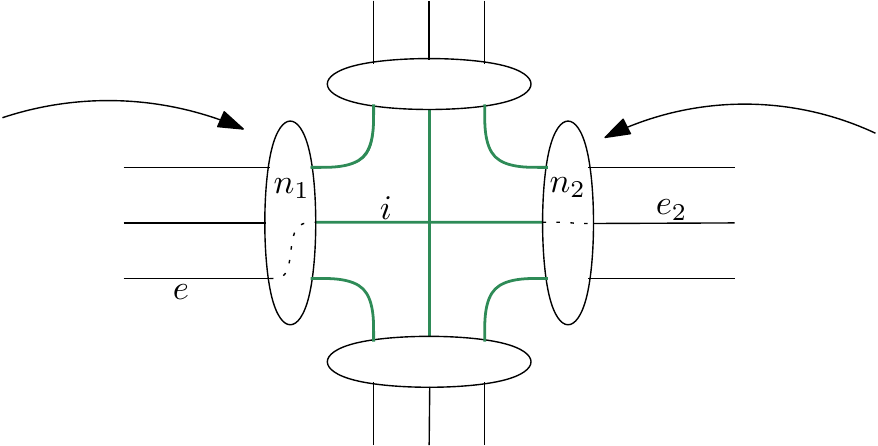}
  \caption{Partial walk being built as in the algorithm. The large arrows indicate the orientation assigned to each set of three parallel external arcs.  In this diagram, $e$ was the starting arc. The choice of $i$ is shown. Note that all three possible choices for $e_2$ are equivalent. Also since the orientation on $e_2$ is ``backwards'', the new partial walk would be $\mathcal{P}\lhd (i,-e_2)$.}
  \label{fig:algorithm}
\end{figure}

However, this is approach is not tractable for any interesting values of $n$, and so
we introduce the following improvements.

\subsection{Limiting the size of walks}
Enumeration algorithms \cite{Burton2004,Burton2007,Regina,Martelli2001,Matveev1998,Matveev2007AlgorithmicTopology} in 3-manifold topology often focus on closed, minimal, irreducible and $\p$-irreducible 3-manifold triangulations.
These properties were all defined in Section \ref{sec:notation}.
For brevity, we say that a triangulation (or manifold decomposition) has such a
property if and only if the underlying manifold has the property.

The following results are taken from \cite{Burton2004}, though in the orientable
case similar results were known earlier by other authors
\cite{Martelli2001,Matveev2007AlgorithmicTopology}.

\begin{lemma}
  {\em (2.1 in \cite{Burton2004})}
No closed minimal triangulation has an edge of degree three that belongs to
three distinct tetrahedra.
\end{lemma}

\begin{lemma}
  {\em (2.3 and 2.4 in \cite{Burton2004})}
  No closed minimal $\p$-irreducible triangulation with $\geq 3$ tetrahedra
  contains an edge of degree $\leq 2$.
\end{lemma}
Given the degree of an edge $e$ of a triangulation is the number of
tetrahedron edges which are identified to form $e$, these results
translate to manifold decompositions as follows.

\begin{corollary}\label{cor:degree1or2}
  No closed minimal $\p$-irreducible manifold decomposition with $\geq 3$
  tetrahedra contains a walk which itself contains less than three external arcs.
\end{corollary}
\begin{corollary}\label{cor:degree3}
  No closed minimal manifold decomposition contains a walk which itself
  contains exactly three internal arcs representing edges on distinct
  tetrahedra (i.e., belonging to three distinct $K_4$ subgraphs).
\end{corollary}

The above results are direct corollaries, as it is simple to translate 
the terms involved and the results are simple enough to implement in an
algorithm. 
In the backtracking algorithm, this means we can implement a check on the number of arcs
in a walk before added the walk to the decomposition.
This is implementable as a constant time check if the length of the current partial walk is stored.

Additionally, for a census of 1-vertex triangulations on $n$ tetrahedra, a manifold decomposition must contain exactly $n+1$ walks.
If the algorithm has completed $k$ walks, then there are $n+1-k$ walks left to complete.
Each such walk must contain at least three external arcs, so if there are less than $3(n+1-k)$ unused external arcs, the current subtree of the search space can be pruned.
\begin{improvement}\label{impro:3x_arcs_remaining}
  If during the enumeration process $k$ walks have been completed and there are less than $3(n+1-k)$ unused external arcs, prune the search tree at this point.
\end{improvement}

We extend this result one step further.
By Corollary \ref{cor:degree3} a closed walk in a manifold decomposition which contains three internal arcs must contain two internal arcs belonging to the same $K_4$, as in Figure \ref{fig:3-walk}.
We modify our algorithm to enumerate all such closed walks first.
Each such walk is either present or absent in any manifold decomposition.
For each possible combination of such walks, we fix said walks and then run the search on the remaining arcs.
All other walks must now contain at least four external arcs, so during the
census on $n$ tetrahedra if the algorithm has completed $k$ walks and there
are less than $4(n+1-k)$ unused external arcs we know that the partial decomposition
cannot be completed to a manifold decomposition.
\begin{improvement}\label{impro:4x_arcs_remaining}
  For each $K_4$ in the given graph, determine if two of its internal arcs can be used together in a walk containing exactly three internal arcs.
  If this is possible, add said walk to the set $S$.
  Then, for each subset $s \subseteq S$, use $s$ as a starting set of walks and attempt to complete the ordered decomposition.
  If during the enumeration process $k$ walks have been completed and there are less than $4(n+1-k)$ unused external arcs, prune the search tree at this point.
\end{improvement}

\begin{figure}
  \centering
  \includegraphics[scale=0.7]{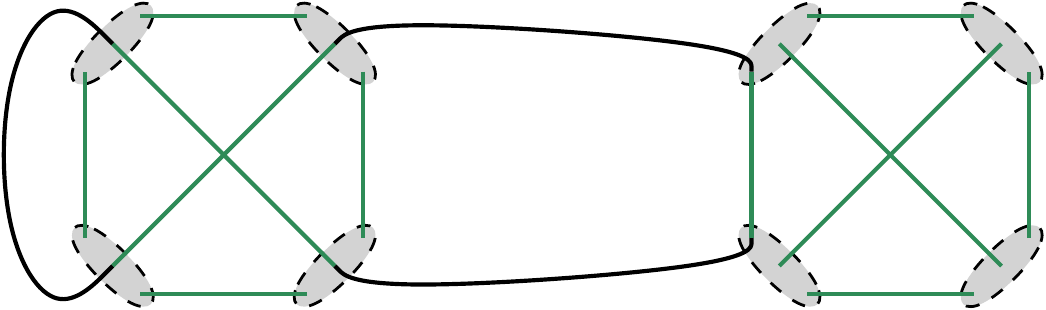}
  \caption{The only possible walk containing 3 internal arcs not all from distinct tetrahedra in a fattened face pairing graph on more than 1 tetrahedron. Only the external arcs used in the walk are shown, other external arcs are not shown.}
  \label{fig:3-walk}
\end{figure}

\subsection{Avoiding cone faces}
For some properties of minimal triangulations, it is not clear that
the corresponding tests can be implemented cheaply.
Here we identify further results from the literature that enable fast
implementations in our setting.
The following was shown in \cite{Burton2004}.

\begin{lemma}
  {\em (2.8 in \cite{Burton2004})}
  Let $T$ be a closed minimal $\p$-irreducible triangulation containing $\geq
  3$ tetrahedra. Then no single face of $T$ has two of its edges identified to
  form a cone as illustrated in Figure \ref{fig:one-face-cone}.
\end{lemma}

\begin{figure}[h]
  \centering
  \includegraphics{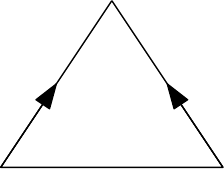}
  \caption{A one-face cone formed by identifying the two marked edges.}
  \label{fig:one-face-cone}
\end{figure}

For manifold decompositions, our translation of this result also requires the
underlying manifold to be orientable in order to give a fast algorithmic test.

\begin{lemma}
  Let $D$ be a closed minimal $\p$-irreducible manifold decomposition of an orientable
  manifold containing $\geq 3$ tetrahedra. Then no walk of $D$ can
  use two parallel external arcs in opposite directions (as seen in Figure \ref{fig:BothDirections}).
\end{lemma}

\begin{figure}
  \centering
  \includegraphics[scale=0.7]{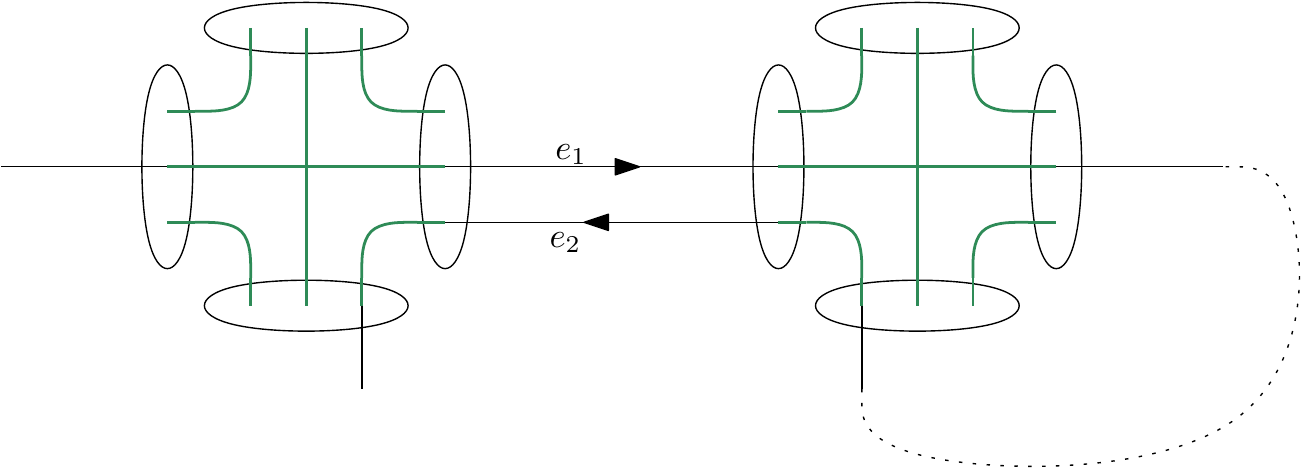}
  \caption{The depicted walk cannot occur in a closed minimal $\p$-irreducible orientable manifold decomposition as external arcs $e_1$ and $e_2$ are used in opposite directions. The dotted lines indicates the walk continues through undrawn parts of the fattened face pairing graph.}
  \label{fig:BothDirections}
\end{figure}

\begin{proof}
Recall that by
our definition, if some walk $P_x$ of a manifold decomposition contains the
sequence of arcs
$(\{v_{i,a},v_{i,b}\},\{v_{i,a},v_{j,c}\})$ then
face $a$ of tetrahedron $i$ is
identified with face $c$ of tetrahedron $j$.
Assume towards a contradiction that we also have the sequence of arcs
$(\{v_{i,a},v_{j,c}\},\{v_{i,a},v_{i,d}\})$ in the
walk $P_x$ somewhere, such that the parallel arcs of the form
$\{v_{i,a},v_{j,c}\}$ are used in the walk in both directions.

\begin{figure}
  \centering
  \subfloat[]{\includegraphics[scale=1.0]{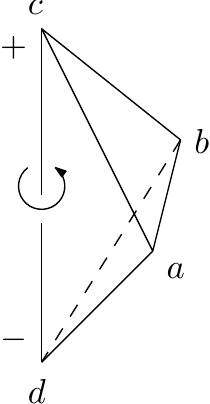}\label{fig:RightHand-a}}
  \qquad
  \subfloat[]{\includegraphics[scale=1.0]{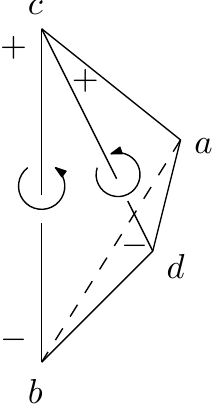}\label{fig:RightHand-b}}
  \caption{A tetrahedron in an oriented triangulation. The circular line with the arrow head indicates a ``right hand'' gripping the edge. In (a) the edge $cd$ is given an orientation. If two parallel external arcs are used in opposite directions (as seen in Figure \ref{fig:BothDirections}) then (b) must occur. Note that since the triangulation is oriented, the labelling of (b) is forced, given the labelling of (a).}
  \label{fig:RightHand}
\end{figure}

Affix some orientation onto the edge of the manifold represented by $P_x$, and
consider the ring of tetrahedra surrounding this edge.
Since we have an
orientable manifold, we can make use of a ``right-hand'' rule. See Figure \ref{fig:RightHand} for a visual aid.
Imagine a right
hand inside tetrahedron $i$, gripping edge $cd$ (represented by $\{v_{i,a},v_{i,b}\}$) such that the
thumb points towards the positive end of the edge and the fingers curl around
the edge so that they leave the tetrahedron through face $a$ (see Figure \ref{fig:RightHand-a}). Since the manifold
is orientable, any time this hand is back inside tetrahedron $i$ it must have
this same orientation. Now since $\{v_{i,a},v_{i,b}\}$ preceded
$\{v_{i,a},v_{j,c}\}$ in the walk and the fingers curl ``out'' through
face $a$ of tetrahedron $i$, if some other arc $\{v_{i,a},v_{i,d}\}$ succeeds
arc $\{v_{i,a},v_{j,c}\}$ then the fingers must curl ``in'' through face
$a$ of tetrahedron $i$ as the hand grips edge $cd$. As yet, this is no contradiction, as the hand is
gripping the one edge of the triangulation, but is therefore gripping many
edges of tetrahedra. However, this necessarily leads to these two edges of
tetrahedra having the same common vertex as their ``positive'' end (see Figure \ref{fig:RightHand-b}). 
Then face $a$ has two edges identified as in \ref{fig:one-face-cone}, contradicting Lemma \ref{lemma:no_single_edge_reverse}.
\end{proof}

This result leads to the following.

\begin{improvement}\label{impro:no_reverse}
When enumerating \emph{orientable} manifold decompsitions,
if an external arc $e$ is to be added to some walk $W$, and $e$ is parallel to
another external arc $e'$ which itself is in $W$, check whether $e$ and $e'$
will be used in opposite directions. If so, do not use $e$ at this point;
instead backtrack and prune the search tree.
\end{improvement}

\subsection{One vertex tests}\label{sec:1vtx}
Definition \ref{definition:manifold_decomp} requires that the associated manifold only
have one vertex. We test this by tracking properties of the vertex links
as the manifold decomposition (i.e., triangulation) is built up.
Specifically, while the manifold decomposition is still being constructed,
no vertex link may be a closed surface.

\begin{figure}
  \centering
  \includegraphics[scale=0.9]{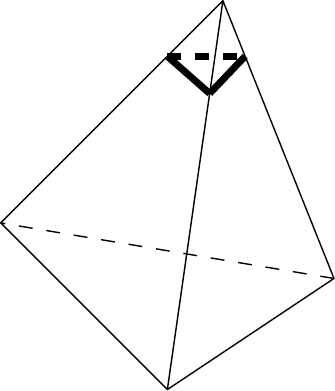}
  \caption{A tetrahedron, with the link of the top vertex drawn in heavier lines. This link, when triangulated, is homeomorphic to a disc. Each of the three heavier lines is a frontier edge.}
  \label{fig:Link}
\end{figure}
Initially, the link of each vertex may be
triangulated as a single triangular face, and therefore has 3 frontier edges.
Each time an external arc is used in a walk, two edges in the
triangulation are identified together, and as a result two frontier edges are identified
together (see Figure \ref{fig:IdentifyEdge}).
\begin{figure}
  \centering
  \includegraphics[scale=0.9]{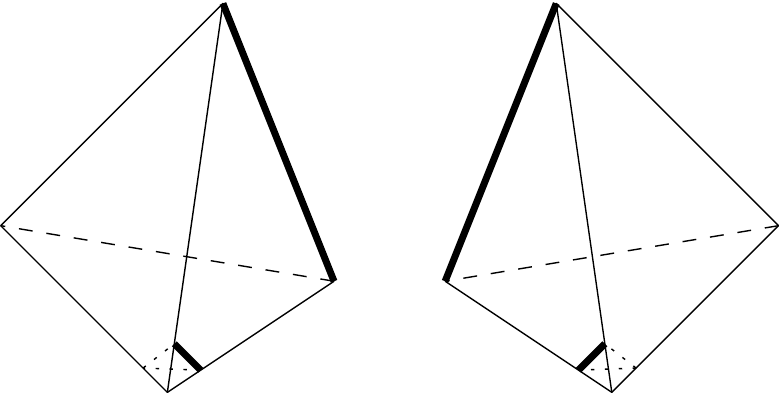}
  \caption{When the two edges of the two tetrahedra (long thick lines) are identified, we also know that the two frontier edges (short thick lines) will be identified.}
  \label{fig:IdentifyEdge}
\end{figure}

The orientation of this
identification is not known, but is also not required. We only require that the
triangulation only have one vertex and we do this by tracking how many frontier edges are in each link.
When frontier edges are identified together, the two edges either belong to the
same link, or to two distinct links.
If the two frontier edges belong to the same link (see Figure \ref{fig:LinkIdentSame}), the number of
frontier edges in the link is reduced by two.
However if the frontier edges belong to two distinct links (see Figure \ref{fig:LinkIdentDiff}), with $l_a$ and $l_b$
frontier edges respectively, the resulting link has $l_a + l_b - 2$ frontier
edges.  Note that after this identification, two links have been joined
together so we must not just track the number of frontier edges, but also which
links have been identified.

\begin{figure}
  \centering
  \subfloat[]{\includegraphics[scale=0.5]{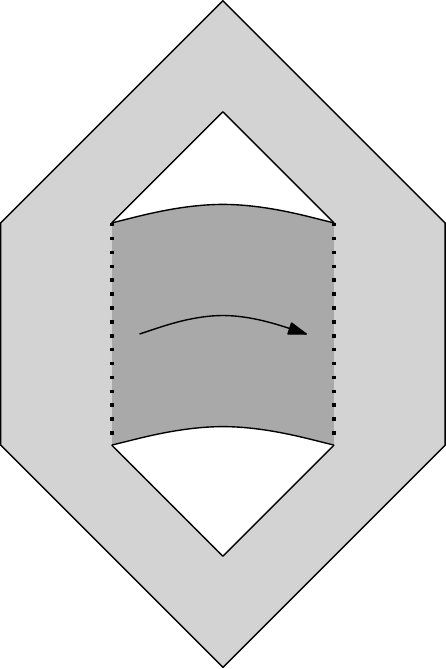}\label{fig:LinkIdentSame}}
  \qquad
  \subfloat[]{\includegraphics[scale=0.8]{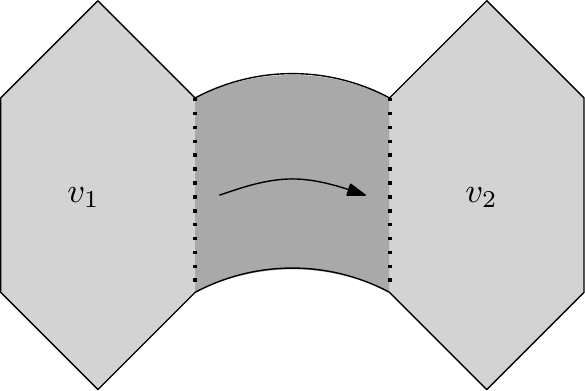}\label{fig:LinkIdentDiff}}
  \caption{Two possibilities when identifying frontier edges of link vertices. The dark grey and the arrow indicate the two edges identified. In (a) the two frontier edges belong to the same vertex link, whereas in (b) the two edges belong to two different vertex links. Note that we are only interested in the number of frontier edges in the link, not its shape or the orientation of any identification.}
  \label{fig:LinkIdent}
\end{figure}

Once a vertex link has no frontier edges, it is a closed surface.
If any other distinct vertex links exist, we know that the triangulation must have more than one vertex, which gives the following.

\begin{improvement}\label{impro:1vtx}
  When building up a manifold decomposition, track how many ``frontier edges'' remain around each vertex link.
  If any vertex links are closed off before the manifold decomposition is completed,
  backtrack and prune the current subtree of the search space.
\end{improvement}

The number of frontier edges of each vertex link, as well as which vertex links are identified together, are tracked via a union-find data structure.
The data structure is slightly tweaked to allow back tracking (see \cite{Burton2007} for details), storing the
number of frontier edges at each node.
For more details on the union-find algorithm in general, see \cite{Sedgewick1992}.

\subsection{Canonicity and Automorphisms}

When running a search, many equivalent manifold decompositions will be found.
These decompositions may differ in the order of the walks found,
or two walks might have different starting arcs or directions.
For example, the two walks $(a,b,c)$ and $(-b,-a,-c)$ are equivalent.
The second starts on a different arc, and traverses the walk backwards, but neither of these change the manifold decomposition.
Additionally, the underlying face pairing graph often has non-trivial automorphism group.

To eliminate such duplication, we only search for {\em canonical} manifold decompositions.
We use the obvious definition for a canonical walk (lowest-index arc is written first and is used in the positive direction).

\begin{definition}
  A walk $P=(x_1,x_2,\ldots,x_m)$ in an ordered decomposition is semi-canonical if
  \begin{itemize}
    \item $x_1 > 0$ ; and
    \item $|x_1| \leq |x_i|$ for $i=2,\ldots,m$.
  \end{itemize}
\end{definition}

\begin{definition}
  A walk $P=(x_1,x_2,\ldots,x_m)$ in an ordered decomposition is canonical if
  \begin{itemize}
    \item $P$ is semi-canonical; and
    \item for any semi-canonical $P'=(x'_1,x'_2,\ldots,x'_m)$ isomorphic (under cyclic permutation of the edges in the path or reversal of orientation) to $P$, either $|x_2| < |x'_2|$ or $|x_2| = |x'_2|$ and $x_2 > 0$.
  \end{itemize}
\end{definition}

This definition of canonical simply says that we always start on the arc with lowest given value, and use said arc in a forwards direction.
If there are two or three such choices, we take the arc which results in the second arc in the walk having lowest value.
If this still leaves us with two choices, we take the walk where we use said second arc in the ``forwards'' direction.
Since there is exactly one internal arc between any two external arcs, we are guaranteed a unique choice at this stage.
\begin{definition}
  Given two walks $P_x=(x_1,\ldots,x_k)$ and $P_y=(y_1,\ldots,y_m)$ in canonical form, we say that $P_x < P_y$ if and only if
  \begin{itemize}
    \item $x_i = y_i$ for $i=1,\ldots,n-1$ and $x_n < y_n$; or
    \item $x_i = y_i$ for $i=1,\ldots,k$ and $k < m$.
  \end{itemize}
\end{definition}

In plainer terms, pairs of arcs from each walk are compared in turn until
one arc index is smaller in absolute value than the other, or
until the end of one walk is reached in which case the shorter walk 
is considered ``smaller''.

\begin{definition}
  A manifold decomposition consisting of walks
  $P_1,P_2,\ldots,P_m$ is considered canonical if:
  \begin{itemize}
    \item $P_i$ is canonical for $i=1,\ldots,m$; and
    \item $P_i< P_{i+1}$ for $i=1,\ldots,m-1$.
  \end{itemize}
\end{definition}

Recall that we may have automorphisms of the underlying face pairing graph to consider.
Each automorphism will relabel the arcs of the labelled fattened face pairing graph.
Each relabelling changes any manifold decomposition by renumbering the arcs in the walks.
We apply each automorphism to a manifold decomposition $\mathcal{D}$ to obtain a new decomposition $\mathcal{D}'$.
Then $\mathcal{D}'$ is made canonical itself (by setting the first external arc in each walk and reordering the walks), and we compare $\mathcal{D}$ and $\mathcal{D'}$.
If $\mathcal{D}' < \mathcal{D}$ then we can discard $\mathcal{D}$ and prune the search tree.

There are two points in the algorithm where we might test for canonical decompositions.
\begin{improvement}\label{impro:CanonEveryArc}
  Every time an external arc is added to a walk, check if the current decomposition is canonical.
  If not, disregard this choice of arc and prune the search tree.
\end{improvement}

\begin{improvement}\label{impro:CanonWalks}
  Every time a walk is completed, check if the current decomposition is canonical.
  If not, disregard this choice of arc and prune the search tree.
\end{improvement}

Unfortunately, checking if a (possibly partial) decomposition is canonical is not computationally cheap.
Experimental results showed that using Improvement \ref{impro:CanonWalks} was significantly faster than using Improvement \ref{impro:CanonEveryArc} as fewer checks for canonicity were made.

\section{Results and Timing}\label{sec:results}

In this section we detail the results from testing the algorithm. We test the
manifold decomposition algorithm and its improvements
from Section \ref{sec:existing-results} against the existing implementation
in {\em Regina}.

{\em Regina} is a suite of topological software and includes state of the art
algorithms for census enumerations in various settings, including non-orientable
and hyperbolic manifolds \cite{burton11-genus,Burton2014Cusped}.
{\em Regina} and its source code are freely available, which facilitates
comparable implementations and fair testing.
{\em Regina} also filters out invalid triangulations as a final stage, which
allows us to test the efficiency of our various improvements by enabling or
disabling them independently.
Like other census algorithms in the literature, {\em Regina} builds triangulations
using the traditional framework by identifying faces two at a time.

We find that while {\em Regina} outperforms our new algorithms overall,
there are non-trivial subcases for which our algorithm runs an order of magnitude faster.
Importantly, in a typical census on 10 tetrahedra, {\em Regina} spends almost half of its running
time on precisely these subcases.
This shows that our new algorithm has an important role to play:
it complements the existing framework by providing a means to remove
some of its most severe bottlenecks.
Section~\ref{sec:graphtests} discusses these cases in more detail.

These observations are, however, in retrospect: what we do not have is a clear
indicator in advance for which algorithm will perform best for any given subcase.

Recall that a full census enumeration involves generating all 4-regular multigraphs, and
then for each such graph $G$, enumerating triangulations with face pairing graph $G$.
In earlier sections we only dealt with
individual graphs, but for the tests here we ran each algorithm on all 4-regular
multigraphs of a given order $n$.

In the following results,
we use the term MD to denote our basic algorithm, using improvements \ref{impro:4x_arcs_remaining}, \ref{impro:1vtx} and \ref{impro:CanonWalks}.
For enumerating orientable manifolds only, we also use Improvement \ref{impro:no_reverse} and denote the corresponding algorithm as MD-o.
Experimentation indicated that Improvement~\ref{impro:1vtx} was computationally expensive,
and so we also tested algorithm MD* (using only Improvements \ref{impro:4x_arcs_remaining} and \ref{impro:CanonWalks}) and algorithm MD*-o (using Improvements \ref{impro:4x_arcs_remaining} \ref{impro:no_reverse} and \ref{impro:CanonWalks}).
Note that these last two algorithms may find ordered decompositions which are
not necessarily manifold decompositions, but we can easily filter these out once
the enumeration is complete.

The algorithms were tested on a cluster of Intel Xeon L5520s running at 2.27GHz.
Times given are total CPU time; that is, a measure of how long the test would
take single-threaded on one core. The algorithms themselves, when run on all 4-regular multigraphs on $n$ nodes, are trivially parallelisable which allows each census to complete much faster by taking advantage of available hardware.

We note that, as expected, the census results are consistent between the
old and new algorithms.

\subsection{Aggregate tests}
In the general setting (where we allow orientable and non-orientable triangulations alike) Table \ref{tab:all_mfolds} highlights that {\em Regina} outperforms MD
when summed over all face pairing graphs. 
The difference seems to grow slightly as $n$ increases, pointing to the possibility that more optimisations in this setting are possible.

We suspect that tracking the orientability of vertex links is giving {\em Regina} an advantage here (see \cite{Burton2007}, Section 5).
Tracking orientability more difficult with ordered decompositions, as the walks
are built up one at a time---each external arc represents an identification of edges, but does not specify the orientation of this identification.
Thus orientability cannot be tested until at least two of any three parallel external arcs are used in walks.

\begin{table}[h]
  \centering
  \caption{Running time of {\em Regina} and the manifold decomposition (MD)
algorithms when searching for manifold decompositions on $n$ tetrahedra.}
\subfloat[][Running times (in seconds) for the general setting.]{
\begin{tabular}{|r|r|r|}\hline
  $n$ & {\em Regina} & MD \\ \hline \hline
  7 & 29 & 80 \\
  8 & 491 & 2453 \\
  9 & 11\,288 & 79\,685 \\
  10 & 323\,530 & 3\,406\,211 \\
  \hline
\end{tabular}
\label{tab:all_mfolds}
}
\qquad \qquad \qquad
\subfloat[][Running times (in seconds) for the orientable setting]{
\begin{tabular}{|r|r|r|}\hline
  $n$ & {\em Regina} & MD-o \\ \hline \hline
  7 & $<1$ & 25 \\
  8 & 147 & 535 \\
  9 & 3\,499 & 13\,161 \\
  10 & 90\,969 & 430\,162 \\
  \hline
\end{tabular}
\label{tab:orientable}
}
\end{table}

We also compare MD-o to {\em Regina}, where we ask both algorithms to only search for orientable triangulations.
Both algorithms run significantly faster than in the general setting
(demonstrating that Improvement \ref{impro:no_reverse} is a significant improvement).
Table \ref{tab:orientable} shows that {\em Regina} outperforms MD-o roughly by a factor of four.
This appears to be constant, and here we expect MD to be comparable to {\em
Regina} after more careful optimisation (such as {\em Regina}'s own algorithm
has enjoyed over the past 13~years \cite{Burton2004,Burton2007}).

To test Improvement \ref{impro:1vtx} (the one-vertex test), we compare MD* and MD*-o against MD and MD-o respectively.
The timing data in Tables \ref{tab:md_star_or} and \ref{tab:md_star} shows that MD* and MD*-o
out-performed MD and MD-o, demonstrating that Improvement \ref{impro:1vtx} actually slows down the algorithm. 
We verified that Improvement \ref{impro:1vtx} is indeed discarding unwanted
triangulations---the problem is that tracking the vertex links is too computationally expensive.
Algorithms MD* and MD*-o instead enumerate these unwanted triangulations and
test for one vertex after the fact, discarding multiple vertex triangulations
after they have been explicitly constructed.
The cost of this is included in the timing results, which confirms that such an
``after the fact'' verification process is indeed faster than the losses incurred by Improvement~\ref{impro:1vtx}.
\begin{table}[h]
  \centering
  \caption{Running times of MD, MD*, MD-o, MD*-o when searching for manifold decompositions on $n$ tetrahedra.}
  \subfloat[][Running times (in seconds) for the general setting.]{
\begin{tabular}{|r|r|r|}\hline
  $n$ & MD & MD* \\ \hline \hline
  7 & 80 & 71 \\
  8 & 2\,453 & 1\,875 \\
  9 & 79\,685 & 58\,743 \\
  10 & 3\,406\,211 & 1\,624\,025 \\
  \hline
\end{tabular}
\label{tab:md_star}
}
\qquad \qquad \qquad
  \subfloat[][Running times (in seconds) for the orientable setting.]{
\begin{tabular}{|r|r|r|}\hline
  $n$ & MD-o & MD*-o \\ \hline \hline
  7 & 25 & 16 \\
  8 & 535 & 446 \\
  9 & 13\,161 & 10\,753 \\
  10 & 430\,162 & 291\,544 \\
  \hline
\end{tabular}
\label{tab:md_star_or}
}
\end{table}

\subsection{Individual graph tests} \label{sec:graphtests}

It is on individual (and often pathological) face pairing graphs
that the new algorithm shines.
Recall that the census enumeration problem requires running an enumeration algorithm on all connected 4-regular multigraphs of a given order.
Table \ref{tab:md_graphs} shows the running time of both {\em Regina} and MD* on a cherry-picked sample of such graphs on 10 tetrahedra.

From these we can see that on some particular graphs, MD* outperforms {\em Regina} by an order of magnitude.
While these graphs were cherry-picked, they do display the shortfalls of {\em Regina}.
There are 48432 4-regular multigraphs on 10 nodes, and it takes {\em Regina} 89.9 CPU-hours to complete this census.
Of these 48432 graphs, 48242 are processed in under 300 seconds each. In
contrast, it takes {\em Regina} 43.6 CPU-hours to process these remaining 190 graphs.
This accounts for 48.5\% of the running time of {\em Regina}'s census on 10 tetrahedra triangulations.

Running these ``pathological'' graphs through MD takes 12.1 CPU-hours, for a saving of 31.5 CPU-hours.
This would reduce the running time of the complete census from 89 hours to 58 hours, a 35\% improvement.

If we find the ideal heuristic which tells us exactly which of {\em Regina} or
MD* will be faster on a given graph, we could always just use the faster algorithm.
This would save 40 hours of computing time for the
10 tetrahedra census, which would turn the running time from 90 CPU-hours down
to 50 CPU-hours, a 44\% improvement.
Further work in this area involves identifying exactly which heuristics and
graph metrics can be used to determine whether {\em Regina} or MD will analyse a given graph faster.

\begin{table}[h]
  \centering
  \caption{Running time in seconds of MD$^*$ and {\em Regina} on particular graphs on 10 nodes. Here ``Task'' identifies the specific graph as being the $i$-th graph produced by {\em Regina}.}
\begin{tabular}{|c|c|c|}\hline
  Task & {\em Regina} & MD$^*$ \\ \hline \hline
  48\,308 & 2476 & 142 \\
  48\,083 & 2487 & 192 \\
  48\,288 & 2164 & 118 \\
  47\,332 & 2141 & 229 \\
  47\,333 & 2003 & 134 \\
  47\,520 & 2083 & 221 \\
  46\,914 & 2108 & 302 \\
  \hline
\end{tabular}
\label{tab:md_graphs}
\end{table}

\bibliography{bib}

\end{document}